\documentclass[12pt]{amsart}

\pdfoutput=1

\usepackage[T1]{fontenc}
\usepackage[a4paper,top=3.6cm,bottom=3.7cm,inner=2.3cm,outer=2.3cm]{geometry}
\usepackage[pdftex]{graphicx,graphics}
\usepackage{color,ifpdf}
\usepackage{amsmath,amssymb,amsfonts,dsfont,enumerate,url}
\usepackage[pdftex]{hyperref}

\newcommand{\R}{\mathbb{R}}
\newcommand{\Z}{\mathbb{Z}}
\newcommand{\N}{\mathbb{N}}
\newcommand{\E}{{\mathbb E}}

\newcommand{\scal}[2]{\langle #1, #2 \rangle}

\newcommand{\grad}{\nabla}

\newcommand{\Xn}{Z^{(n)}}

\renewcommand{\P}{\mathbb{P}}
\newcommand{\PP}{\mathsf{P}}
\newcommand{\EE}{\mathsf{E}}
\newcommand{\TT}{\mathsf{T}}
\newcommand{\LL}{\mathsf{L}}

\RequirePackage{yhmath} 
 \renewcommand\widering[1]{\ring{#1}}
\newtheorem{thm}{Theorem}[section]

\newtheorem{cor}{Corollary}[section]
\newtheorem{lem}{Lemma}[section]
\newtheorem{con}{Conjecture}[section]
\newtheorem{rem}{Remark}[section]
\theoremstyle{definition}
\newtheorem{defi}[thm]{Definition}
\numberwithin{equation}{section}

\graphicspath{{artwork/}}

\title[Large deviations for the local fluctuations of random walks]{Large deviations for  the local fluctuations of random walks and new insights into the ``randomness'' of Pi}

\begin{document}

\author{Julien Barral}
\address[J. Barral]{LAGA (UMR 7539), D\'epartement de Math\'ematiques, Institut Galil\'ee, Universit\'e
 Paris 13, 99 avenue Jean-Baptiste Cl\'ement , 93430  Villetaneuse, France}
\thanks{Corresponding author: Julien Barral ({\tt barral@math.univ-paris13.fr}). The authors were partially supported by the French National Research Agency  Project ``DMASC''}
\email{barral@math.univ-paris13.fr}
\author{Patrick Loiseau}
\address[P. Loiseau]{EURECOM, 2229 route des cr\^etes, BP 193, F-06560 Sophia-Antipolis cedex, France}
\email{patrick.loiseau@eurecom.fr}

\begin{abstract}
We establish large deviations properties valid for almost every sample path of a class of stationary mixing processes $(X_1,\dots, X_n,\dots)$. These properties are inherited from those of  $S_n=\sum_{i=1}^nX_i$ and describe how the local fluctuations  of almost every realization of $S_n$ deviate from the almost sure behavior. These results apply to the fluctuations of Brownian motion, Birkhoff averages on hyperbolic dynamics, as well as branching random walks. Also, they lead to new insights into the ``randomness'' of the digits of expansions  in integer bases of Pi.  We formulate a new conjecture, supported by numerical experiments, implying the normality of Pi. 
\end{abstract}


\maketitle

\section{Introduction}
\label{sec.introduction}

Given a sequence of  i.i.d. real valued random variables $(X_n)_{n\ge 1}$, large deviations theory provides a precise estimate of the probability that the random walk $S_n=\sum_{i=1}^nX_i$ deviates from its almost sure asymptotic behavior, as long as $X$ possesses finite exponential moments on a non trivial domain. In particular, if $\Lambda(\lambda)=\log \E(\exp(\lambda X_1))$ is finite over an interval $\mathcal D_\Lambda$ whose interior contains 0 then (see Cramer's theorem in \cite{DZ} Ch 2.2 or G\"artner-Ellis' theorem at the end of this section) 
\begin{equation}\label{LDclassic}
\forall \ x\in \Lambda'(\widering{\mathcal{D}}_\Lambda), \ \lim_{\epsilon\to 0}\lim_{n\to\infty}\frac{1}{n}\log\mu_n([x-\epsilon,x+\epsilon])=-\Lambda^*(x),
\end{equation}
where $\mu_n$ is the  distribution of $S_n/n$ and $\Lambda^*(x)=\sup_{\lambda\in \R} \{ \lambda x-\Lambda(\lambda)\}$. Notice that the case $x=\Lambda'(0)=\E(X_1)$ corresponds to the almost sure asymptotic behavior of $S_n(\omega)$ given by the strong law of large numbers: $S_n(\omega)/n\to\E(X_1)$ as $n\to\infty$, and $\Lambda^*(\E(X_1))=0$. 

\medskip

In this paper, we show that this large deviation principle (LDP) is transfered to almost every path of the random walk, though the behavior of $S_N(\omega)/N$ is prescribed by the strong law of large numbers.  To see this, we look at the deviations from this behavior over the blocks $(X_{(j-1)n+1}(\omega),\cdots,X_{jn}(\omega))$ of length $n\ll N$ picked up in $(X_i(\omega))_{1\le i\le N}$. Specifically, we define 
$$
\Delta S_n(j,\omega)=S_{jn}(\omega)-S_{(j-1)n}(\omega)
$$
and for $N=k(n)\cdot n$ with $k(n)\to\infty$ as $n\to\infty$, we seek a LDP providing the almost sure asymptotic  behavior of $\displaystyle \#\big \{1\le j\le k(n):  \Delta S_n(j,\omega)\in [n(x-\epsilon),n(x+\epsilon)]\big\}$ and its possible connection with \eqref{LDclassic}. Such a LDP would describe the local fluctuations of $S_n$. We shall obtain the following result as a special case of a more general statement (Theorem~\ref{th3}). We consider the random sequence of Borel measures $(\mu_n^\omega)_{n\ge 1}$ on $\R$ defined as 
$$
\mu_n^\omega(B)=\frac{ \#\big \{1\le j\le k(n):  \Delta S_n(j,\omega)/n\in B\big\}}{k(n)}\quad(\text{for every Borel set } B),  
$$
as well as their logarithmic generating functions 
$$
\Lambda_n^\omega(\lambda)= \frac{1}{n}\log \int_\R\exp (n\lambda x)\,{\rm d}\mu^\omega_n(x)=\displaystyle \frac{1}{n}\log \Big (\frac{1}{k(n)}\sum_{j=1}^{k(n)}\exp (\lambda\Delta S_{n}(j,\omega))\Big ) \quad (\lambda\in\R).
$$

\begin{thm}\label{th1} Let $(k(n))_{n\ge 1}$ be a sequence of positive integers. Let  $\lambda_0\in\widering{\mathcal{D}}_\Lambda$ and denote $\Lambda'(\lambda_0)$ as $x_0$.  

\begin{enumerate}\item If $\displaystyle \liminf_{n\to\infty} \frac{\log k(n)}{n}>\Lambda^*(x_0)$ then there exists a neighborhood $U $ of $\lambda_0$ in $\widering{\mathcal{D}}_\Lambda$ such that, with probability 1,  for all $\lambda\in U$
\begin{equation}\label{laplace}
\lim_{n\to\infty}  \Lambda_n^\omega(\lambda)=\Lambda(\lambda),
\end{equation}
and  
\begin{eqnarray*}
\lim_{\varepsilon\to 0}\lim_{n\to\infty}\frac{1}{n}\log\mu^\omega_n([x_0-\epsilon,x_0+\epsilon])=-\Lambda^*(x_0).
\end{eqnarray*}
\item If $\displaystyle \limsup_{n\to\infty} \frac{\log k(n)}{n}<\Lambda^*(x_0)$  and $\epsilon$ is small enough, with probability 1, for $n$ large enough the set  $\big \{1\le j\le k(n): \Delta S_n(j,\omega)\in [n(x_0-\epsilon),n(x_0+\epsilon)]\big\}$  is empty. 

\item If $\displaystyle \lim_{n\to\infty} \frac{\log k(n)}{n}=\Lambda^*(x_0)$ then, with probability 1, for all $t\ge 1$ we have 
\begin{equation}\label{laplacebis}
\lim_{n\to\infty}  \Lambda_n^\omega(t\lambda_0)=\Lambda(\lambda_0)+(t-1)\lambda_0 x_0. 
\end{equation}

\end{enumerate}
\end{thm}
\begin{rem}
{\rm  (1) The previous result will be extended to weakly dependent sequences such that $\Lambda(\lambda)=\lim_{n\to\infty} \frac{1}{n}\log \E(\exp(\lambda S_n))$ exists as $n$ tend to $\infty$ for each $\lambda$ in an open interval. For such sequences, one also has a strong law of large numbers so that for each $n_0\ge 1$, one has 
$$
\lim_{k\to\infty}  \displaystyle \frac{1}{n_0}\log \Big (\frac{1}{k}\sum_{j=1}^{k}\exp (\lambda\Delta S_{n_0}(j,\omega))\Big )=\frac{1}{n_0}\log \E(\exp(\lambda S_{n_0})),
$$
hence
\begin{equation}\label{SLLN}
\lim_{n\to\infty}\lim_{k\to\infty}  \displaystyle \frac{1}{n}\log \Big (\frac{1}{k}\sum_{j=1}^{k}\exp (\lambda\Delta S_n(j,\omega))\Big )=\Lambda(\lambda).
\end{equation}
In Theorem~\ref{th2} we give, in terms of the growth of $\log(k)/n$, a fine measurement of  how $\frac{1}{n}\log \Big (\frac{1}{k}\sum_{j=1}^{k}\exp (\lambda\Delta S_n(j,\omega))\Big )$ is close to $\Lambda(\lambda)$. 

\medskip

\noindent
(2) Theorem~\ref{th1} cannot be obtained as a consequence of \eqref{SLLN}.
}
\end{rem}

Let us show how Theorem~\ref{th1} applies to the description of the dyadic expansion of real numbers. For $t\in [0,1]$ and $i\ge 1$ denote by $t_i$ the $i^{th}$ digit of the dyadic expansion of $t$ (the dyadic points,  which have two expansions, are of no influence in our study): $t=\sum_{i\ge 1} t_i 2^{-i}$. Let $\P_p$ stand for the Bernoulli product of parameter $p\in (0,1)$,  so that the $X_i(t)=t_i$ are i.i.d. Bernoulli variables of parameter $p$ under $\P_p$ ($\P_{1/2}$ is the Lebesgue measure).  By the strong law of large numbers, for $\P_p$-almost every $t$, $\lim_{N\to\infty}\sum_{i=1}^Nt_i/N=p$. Here, $\mathcal {D}_\Lambda=\R$, $\Lambda(\lambda)=\log(1-p+p\exp(\lambda))$, $\Lambda'(\R)=(0,1)$, and ${\Lambda}^{*}(x)=x\log(x/p)+(1-x)\log((1-x)/(1-p))=H(\P_{x}|\P_p)$ for all $x\in (0,1)$. As a consequence of  Theorem~\ref{th1}(1), if $\displaystyle \liminf_{n\to\infty} \frac{\log k(n)}{n}>-\min(\log (p),\log (1-p))$, for $\P_p$-almost every $t$, for all $x\in(0,1)$
\begin{eqnarray*}
\lim_{\varepsilon\to 0}\lim_{n\to\infty}\frac{1}{n}\log \Big (\frac{1}{k(n)}\displaystyle \#\Big \{1\le j\le k(n): \Big ( \sum_{(j-1)n<i\le jn}t_i\Big )\in [n(x-\epsilon),n(x+\epsilon)]\Big\}\Big )=-\Lambda^*(x).
\end{eqnarray*}

\medskip

Once one has such a result, it is very tempting to investigate whether or not it highlights questions related to the distribution of digits for numbers suspected to be normal in a given integer basis $m\ge 2$, i.e. such that for every $n_0\ge 1$, for every finite sequence $(\varepsilon_1,\dots, \varepsilon_{n_0})\in \{0,\dots, m-1\}^{n_0}$, the frequency of the occurrence of  $(\varepsilon_1,\dots, \varepsilon_{n_0})$ in the $m$-adic expansion of $t=\sum_{i\ge 1}t_im^{-i}$ is equal to $m^{-n_0}$, i.e.
\begin{equation}\label{frequency}
\lim_{k\to\infty}\frac{1}{k}\#\{1\le i\le k: (t_i,\dots,t_{i+n_0-1})=(\varepsilon_1\cdots \varepsilon_{n_0})\}= m^{-n_0}.
\end{equation}
 Indeed, for such numbers like the fractional part of Pi, numerical experiments support the conjecture that \eqref{frequency} holds, showing that these numbers  share statistical properties with almost every realization of a sequence $X$ of independent random variables uniformly distributed in $\{0,\dots, m-1\}$, and in this sense are ``random''.  The recent discovery of the so-called BBP algorithm \cite{BBP}, to compute the $n^{th}$ digit without computing the preceding digits, has opened new perspectives  on this question \cite{Bailey}. Theorem~\ref{th1} leads to strengthen the conjecture about the  ``randomness'' of Pi: the sequence of digits of Pi in a given integer basis obeys the same large deviations properties as almost every realization of $X$ (see conjecture~\ref{conjec} for a precise statement). This conjecture, which implies the normality of Pi, is supported by numerical experiments presented in Section~\ref{Pi}. 
  
\medskip

We will obtain extensions of Theorem~\ref{th1} valid for a class of $\R^d$-valued stationary mixing processes. We will also obtain a general result concerning the transfer of  LDPs valid for random walks taking values in a separable normed vector space to LDPs valid for the local fluctuations of almost every  realization of such random walks. These results are stated in Section~\ref{State} and illustrated with several natural examples in Section~\ref{examples}, namely  Brownian motion, Birkhoff sums on symbolic spaces and some of their geometric realizations, branching random walks on Galton-Watson trees, and Poissonian random walks on Poisson point processes.  The proofs of the main results are given in Sections~\ref{Proofs} and~\ref{pfbrown}. 


We end this section by recalling general facts about large deviations theory.

\subsection*{General facts about large deviations theory}\label{state0} 

Let $\mathcal Y$ be a topological space and $B_{\mathcal Y}$ stand for the completed Borel $\sigma$-field. Let $(\mu_n)_{n\ge 1}$ be a sequence of probability measures on $(\mathcal Y,B_{\mathcal Y})$. Let $I:\mathcal Y\to [0,\infty]$ be a lower semi-continuous function. The domain of $I$ is defined as $\mathcal{D}_I=\{x:I(x)<\infty\}$. 

One says (see \cite{DZ} Ch. 1.2) that $(\mu_n)_{n\ge 1}$ satisfies in $\mathcal Y$ the LDP with rate function $I$ if for all set $\Gamma\in B_{\mathcal Y}$. 
\begin{equation}\label{ldpgen}
-\inf_{x\in \widering\Gamma }I(x)\le \liminf_{n\to\infty} \frac{1}{n}\log \mu_n(\Gamma)\le \limsup_{n\to\infty}  \frac{1}{n}\log \mu_n(\Gamma)\le -\inf_{x\in \overline\Gamma }I(x).
\end{equation}
The function $I$ is said to be a  good rate function if, moreover, for any $\alpha \in \R_+$ the level set $\{x\in\mathcal Y:I(x)\le \alpha\}$ is compact. 

One says that $(\mu_n)_{n\ge 1}$ satisfies in $\mathcal Y$ the weak LDP with rate function $I$ if the upper bound in \eqref{ldpgen} holds when $\Gamma$ is a compact subset of $\mathcal Y$.

The sequence $(\mu_n)_{n\ge 1}$ is said to be exponentially tight if for every $\alpha<\infty$ there exists a compact set $K_\alpha\subset\mathcal Y$ such that $\limsup_{n\to\infty} \frac{1}{n}\log\mu_n(K_\alpha^c)\le -\alpha$. In this case, if $(\mu_n)_{n\ge 1}$ satisfies the weak LDP with rate function $I$ then it satisfies the LDP with good rate function $I$ (see \cite{DZ} Lemma 1.2.18).

Large deviation principles have been derived successfully for various stochastic processes  (see, e.g. \cite{Ellis, S,V,DZ}) as well as for dynamical systems (see, e.g. \cite{OP,OP2,Young}).

\subsection*{The G\"artner-Ellis theorem}  It is sometimes possible to derive, or relate, such a principle with the logarithmic generating functions of the measures $\mu_n$ whenever $\mathcal Y$ is a topological vector space. In this paper, when we use such a connection, we take $\mathcal Y=\R^d$ ($d\ge 1$). Then, the main tool is the G\"artner-Ellis theorem whose statement requires the following assumptions and definitions (see \cite{DZ} Ch. 2.3, and \cite{DZ} Ch. 4.5.3 for a version in topological vector spaces). Let $\scal{\cdot}{\cdot}$ stand for the canonical scalar product on $\R^d$.  Let $(\mu_n)_{n\ge 1}$ be a sequence of probability measures on $(\R^d,B_{\R^d})$. For each $n\ge 1$ let
$$
\Lambda_n({\lambda})=\frac{1}{n}\log \int_{\R^d} \exp (n\scal{\lambda}{x})\, {\rm d}\mu_n(x)\quad (\lambda\in \R^d).
$$

Assume {\bf (A):} For each $\lambda\in\R^d$, $
\Lambda({\lambda})=\lim_{n\to\infty}\Lambda_n({\lambda})$ 
exists as an extended real number. Further, the origin belongs to the interior of $\mathcal{D}_\Lambda=\{\lambda\in\R^d:\Lambda(\lambda)<\infty\}$. 

\medskip

The Fenchel-Legendre transform of $\Lambda$ is defined as 
$$
\Lambda^*(x)=\sup\{\scal{\lambda}{x}-\Lambda (\lambda):\lambda\in\R^d\}\quad (x\in\R^d),
$$
and one sets $\mathcal{D}_{\Lambda^*}=\{x\in\R^d: \Lambda^*(x)<\infty\}$. 
\begin{defi}
$y\in\R^d$ is an exposed point of $\Lambda^*$ if for some $\lambda\in\R^d$ and all $x\neq y\in \R^d$, $\scal{\lambda}{y}-\Lambda^*(y)>   \scal{\lambda}{x}-\Lambda^*(x)$. Such a $\lambda$ is called an exposing hyperplane.
\end{defi}
\begin{defi} Let $L:\R^d\to \R\cup\{\infty\}$ and $\mathcal D_L=\{\lambda\in\R^d: L(\lambda)<\infty\}$. The function $L$ is said essentially smooth if:
\begin{enumerate}[(a)]
\item $\widering {\mathcal{D}}_L$ is non-empty.
\item $L$ is differentiable throughout $\widering {\mathcal{D}}_L$. 
\item $L$ is steep, namely, $\lim_{n\to\infty}|\grad L(\lambda_n)|=\infty$ whenever $(\lambda_n)_{n\ge 1}$ is a sequence in  $\widering {\mathcal{D}}_L$ converging to a boundary point of $\widering {\mathcal{D}}_L$.
\end{enumerate}
\end{defi}

\begin{rem}
Corollary 25.1.2 of \cite{Roc} ensures that the exposed points of $\Lambda^*$ are precisely those $x$ of the form $\grad\Lambda (\lambda)$ for some $\lambda\in \widering {\mathcal{D}}_\Lambda$. Moreover, Theorem 25.5 of \cite{Roc} ensures  that $\Lambda$ is differentiable almost everywhere in  $\lambda\in \widering {\mathcal{D}}_\Lambda$, and if it is differentiable everywhere in $\lambda\in \widering {\mathcal{D}}_\Lambda$, then it is $C^1$. 
\end{rem}

\begin{thm}{\bf (G\"artner-Ellis)}\label{G-E} Under the above assumption {\bf (A)}:
\begin{enumerate}
\item For any closed set $F\subset \R^d$,
$$
\limsup_{n\to\infty} \frac{1}{n}\log \mu_n(F)\le  -\inf_{x\in F}\Lambda^*(x).
$$
\item For any open set $G\subset\R^d$, 
$$
\liminf_{n\to\infty} \frac{1}{n}\log \mu_n(G)\ge -\inf_{x\in G \cap \mathcal F}\Lambda^*(x),
$$
where $\mathcal{F}$ is the set of exposed points of $\Lambda^*$ whose exposing hyperplane belongs to $\widering {\mathcal{D}}_\Lambda$.

\item If $\Lambda$ is essentially smooth and lower semi-continuous, then the LDP with good rate function $\Lambda^*$ holds in $\R^d$ for $(\mu_n)_{n\ge 1}$. 
\end{enumerate}
\end{thm}

A local version of this result is the following. It can be deduced from the proof of G\"artner-Ellis' theorem. Throughout, $B(\lambda,r)$ stands for the closed ball of center $\lambda$ and radius $r$.
\begin{thm}\label{LDPloc}
Suppose that $\Lambda=\lim_{n\to\infty}\Lambda_n$ exists and is finite over an open set $\mathcal{D}$. At any point $\lambda$ of $\mathcal D$ at which $\Lambda$ is differentiable, one has 
$$
\lim_{\epsilon\to 0}\lim_{n\to\infty}\frac{1}{n}\log \mu_n(B(\grad\Lambda(\lambda),\epsilon))=-\Lambda^*(\grad \Lambda (\lambda)).
$$
\end{thm}
\begin{rem}{\rm 
In Theorem~\ref{LDPloc} we do not require that $0\in\mathcal{D}$. This is because one goes back to this assumption by a standard reduction, systematically used in the proof of the lower bound part of G\"artner-Ellis' theorem, as follows. Fix any $\lambda_0\in\mathcal {D}$ and replace $\mu_n$ by the measure ${\rm d}\widetilde \mu_n (x)=\exp (n\scal{\lambda_0}{x} -n\Lambda_n(\lambda_0)){\rm d}\mu_n (x)$. Then replace $\Lambda_n({\lambda})$ by $\widetilde \Lambda_n(\lambda)=\frac{1}{n}\log \int_{\R^d} \exp (n\scal{\lambda}{x})\, {\rm d}\widetilde \mu_n(x)=\Lambda_n({\lambda+\lambda_0})-\Lambda_n({\lambda_0})$.
}\end{rem}



\section{Large deviations principle for local fluctuations of random walks}\label{State}

We need a few notation and definitions related to the notion of weak dependence. 

\medskip

Let  $(\Omega,\mathcal A, \P)$ be a probability space. Given two sub-$\sigma$-algebras $\mathcal{U}$ and $\mathcal{V}$ of $\mathcal A$, their  $\alpha$-mixing coefficient is defined as (see \cite{Rio00a} for a detailed account):
\begin{equation}
\label{eq.defphi}
\alpha (\mathcal{U}, \mathcal{V}) = \sup \left\{   \left| \P(U)\P (V) - \P (U \cap V) \right|: U \in \mathcal{U}, V \in \mathcal{V}  \right\}.
\end{equation}

\medskip

Let $(E,\mathcal T)$ be a measurable space. We consider  $X = (X_i)_{i\in\N_+}$, a stationary process defined on $(\Omega,\mathcal A, \P)$ and taking values in $E$.

For each $n\ge 1$, we define 
\begin{equation}\label{Xn}
X^{(n)}=(X^{(n)}_i)_{i\ge1}=((X_{n(i-1)+1},\dots,X_{in}))_{i\ge 1}.
\end{equation}
Let $T$ denote the shift operation on $E^{\N_+}$: 
\begin{equation*}
T(x_1,x_2,\dots) = (x_2,\cdots ).
\end{equation*}

We will assume that $X$ satisfies some mixing properties. 

The Rosenblatt \cite{Rose} mixing coefficients $(\alpha_{X,m})_{m\ge 0}$ of the sequence $(X_i)_{i\in \N_+}$ are defined as: 

\begin{equation}
\alpha_{X,0}=1/2\text{ and }\alpha_{X,m}= \sup \{\alpha( \sigma(X_k: k\le i), \sigma(X_{j}: j\ge i+m)):i\ge 1\}\quad \text{for }m \geq 1.
\end{equation}

Then, for $u\in [0,1]$, one defines 
$$
\alpha^{-1}_X(u)=\inf\{m: \alpha_{X,m}\le u\}. 
$$

\medskip

For each $n\ge 1$ we let $S_n\Phi (X)$ stand for a measurable function of $(X,TX,\dots,T^{n-1}X)$ taking values in a normed vector space $(\mathcal Y,\|\,\|)$ endowed with the completed Borel $\sigma$-field $B_{\mathcal Y}$. A typical example will be the Birkhoff sums $\sum_{j=0}^{n-1}\Phi(T^jX)$ associated with a measurable function $\Phi: E^{\N_+} \rightarrow \mathcal Y$. 

\medskip

For each $n\ge 1$, denote by $\mu_n$ the distribution of the random variable $S_n\Phi (X)/n$ (viewed under $\mathbb{P}$).

If $\mathcal Y=\R^d$, we define the sequence of logarithmic moment generating functions
\begin{equation}
\label{eq.defPn}
\Lambda_n (\lambda) =  \frac{1}{n} \log \E\exp\big ( \scal{\lambda}{S_n \Phi(X)} \big ) =  \frac{1}{n} \log \int_{\R^d}\exp(n\scal{\lambda}{x})\,{\rm d}\mu_n(x)\quad(\lambda\in \R^d).
\end{equation}

\medskip

Our results will use assumptions among the following. They are divided into three types.

\medskip
 
{\bf (1) Large deviations properties.}

\medskip

{\bf (A1)} The sequence $(\mu_n)_{n\ge 1}$ satisfies in $\mathcal Y$ the LDP with rate function denoted by $I$. 

{\bf (A1')} $\mathcal Y=\R^d$, and there exists a convex open set $\mathcal{D}$ in $\R^d$ such that 
$$
\Lambda(\lambda)= \lim_{n\to\infty} \Lambda_n(\lambda)
$$
exists and is finite for each $\lambda\in \mathcal{D}$.  
 

{\bf (A1'')} $\mathcal Y=\R^d$, and for each $\lambda\in \R^d$, $\Lambda(\lambda)=\lim_{n\to\infty} \Lambda_n(\lambda)$ exists as an extended real number, and the origin belongs to the interior of $\mathcal{D}_\Lambda=\{\lambda: \Lambda(\lambda)<\infty\}$.

\bigskip

{\bf (2) Mixing properties.}

\medskip

{\bf (A2)} $M_h=\int_0^1(\alpha_{X}^{-1}(u))^h\, {\rm d}u<\infty$ for all $h>0$. 

{\bf (A2')} There exists $\gamma>0$ and $\theta>0$ such that $\alpha_{X,m}=O(\exp(-\gamma m^{\theta}))$.  

\bigskip

{\bf (3) Approximation properties.}

\medskip

{\bf (A3)} There exists a sequence $(S_n\Phi_n)_{n\ge 1}$ of functions from $E^{\N_+}$ to $\mathcal Y$ such that each $S_n\Phi_n$ depends only on the $2n$ first coordinates and 
$$
\delta_n=\sup_{z\in E^{\N_+}}\|S_n\Phi (z)-S_n\Phi_n(z)\|/n=o(1)\text{ as $n\to\infty$}.
$$

\medskip

Condition {\bf (A3)} holds in particular if $S_n\Phi(X)$ is given by the Birkhoff sums of a function $\Phi$ defined on $E^{\N_+}$ and there exists a sequence of functions $(\Phi_n)_{n\ge 1}$ defined on $E^{\N_+}$ depending on the $n$ first coordinates only, such that $\sup_{z\in E^{\N_+}}\|\Phi (z)-\Phi_n(z)\|=o(1)$ as $n$ tends to $\infty$. If there exists an integer $p\ge 1$ such that $\Phi$ depends on the $p$ first coordinates only, then  one can take $\Phi_n=\Phi$ for $n$ large enough, and then $\delta_n=0$.

\bigskip

Now, we introduce the family of (random) probability measures for which we obtain large deviations results. 

We fix an increasing sequence of positive integers $(k(n))_{n\ge 1}$, and for every $\omega\in \Omega$ and $n\ge 1$, we define
$$
\mu^\omega_n=\frac{1}{k(n)} \sum_{j=1}^{k(n)} \delta_{x_{n,j}(\omega)}\text{ with } x_{n,j}(\omega)=S_n\Phi(T^{(j-1)n}X(\omega))/n.
$$
In other words, for each Borel set $B\subset \R^d$, we have 
$$
\mu^\omega_n(B)=\frac{\displaystyle \#\big \{1\le j\le k(n): S_n \Phi (T^{(j-1)n}X(\omega))/n\in B\big\}}{k(n)}.
$$
If $\mathcal {Y}=\R^d$, we also define 
$$
\Lambda^\omega_n(\lambda)=\frac{1}{n}\log \int_{\mathbb{R}^d}\exp(n\scal{\lambda}{x}){\rm d}\mu_n^\omega (x).
$$

\medskip

Now we start with results on the direct transfer of the LDP for $S_n\Phi(X)$ to the LDP  for the local fluctuations of almost every realization of $S_n\Phi(X)$. 


\begin{thm}\label{ASLDP} Assume {\bf (A1-3)}.
\begin{enumerate}
\item Let $x\in \mathcal{D}_I$. If $\displaystyle \liminf_{n\to\infty} \frac{\log k(n)}{n}>I(x)$ then, with  probability~1, 
$$
\displaystyle \lim_{\epsilon\to 0^+}\liminf_{n\to\infty} \frac{1}{n}\log\mu^\omega_n(B(x,\epsilon))=\lim_{\epsilon\to 0^+}\limsup_{n\to\infty} \frac{1}{n}\log\mu^\omega_n(B(x,\epsilon))=-I(x).
$$
If $\displaystyle \limsup_{n\to\infty} \frac{\log k(n)}{n}<I(x)$ then there exists $\epsilon>0$ such that, with probability 1, for $n$ large enough the set  $\big \{1\le j\le k(n): S_n \Phi (T^{(j-1)n}X(\omega))/n\in B(x,\epsilon)\big\}$  is empty.

\item Let  $x\in\mathcal Y\setminus \mathcal{D}_I$.
With probability 1, 
$$
\lim_{\epsilon\to 0^+} \limsup_{n\to\infty}\frac{1}{n}\log \mu^\omega_n(B(x,\epsilon))=-I(x)=-\infty
.$$

\medskip

\item If $(\mu_n)_{n\ge 1}$ is exponentially tight, then so is $(\mu_n^\omega)_{n\ge 1}$ almost surely.
\end{enumerate}
\end{thm}

\begin{thm} \label{corASLDP}
Assume $(\mathcal Y,\|\,\|)$ is separable, as well as {\bf (A1-3)}.
Suppose that $\sup_{x\in{\mathcal{D}}_I} I(x)<\infty$ and $\displaystyle \liminf_{n\to\infty} \frac{\log k(n)}{n}> \sup_{x\in {\mathcal{D}_I}}I(x)$, or that $\displaystyle \lim_{n\to\infty} \frac{\log k(n)}{n}=\infty$. With probability 1, $(\mu_n^\omega)_{n\ge 1}$ satisfies in $\mathcal Y$ the weak LDP with rate function $I$. If, moreover, $(\mu_n)_{n\ge 1}$ is exponentially tight, then  $(\mu_n^\omega)_{n\ge 1}$ satisfies in $\mathcal Y$ the LDP with good rate function $I$. 
\end{thm}

Next we give results concerning the transfer of convergence properties for $\Lambda_n$ to convergence properties for $\Lambda_n^\omega$. It is worth mentioning that under the assumptions of Theorem~\ref{corASLDP}, if one has additional information like $\|S_n\Phi\|_\infty =O(n)$, then Varadhan's integral lemma (see \cite{DZ} Th. 4.3.1) together with Theorem~\ref{corASLDP} directly provides the almost sure pointwise convergence of $\Lambda_n^\omega$ to $\Lambda$ as $n\to\infty$.

\begin{thm}
\label{th3}
Assume {\bf (A1')} and {\bf (A2-3)}. Let $\mathbf{\lambda}_0\in \widering{\mathcal{D}}$ at which $\Lambda$ is differentiable and denote $\grad \Lambda(\lambda_0)$ as $x_0$.

\begin{enumerate}
\item If $\displaystyle \liminf_{n\to\infty} \frac{\log k(n)}{n}>\Lambda^*(x_0)$ then there exists $r>0$ such that $B(\lambda_0,r)\subset\mathcal D$ and, with probability 1,  $\Lambda^\omega_n$ converges uniformly  to $\Lambda$ over $B(\lambda_0,r)$. 

\medskip

\item If $\displaystyle \liminf_{n\to\infty} \frac{\log k(n)}{n}>\Lambda^*(x_0)$  then, with probability 1,  
\begin{equation}\label{LD1}
\lim_{\varepsilon\to 0}\lim_{n\to\infty}\frac{1}{n}\log\mu_n^\omega\big (B(x_0,\epsilon)\big )=-\Lambda^*(x_0).
\end{equation}
If $\displaystyle \limsup_{n\to\infty} \frac{\log k(n)}{n}<\Lambda^*(x_0)$, there exists $\epsilon>0$ such that, with probability 1, for $n$ large enough the set  $\big \{1\le j\le k(n): S_n \Phi (T^{(j-1)n}X(\omega))/n\in B(x_0,\epsilon)\big\}$  is empty. 

\item If $\displaystyle \lim_{n\to\infty} \frac{\log k(n)}{n}=\Lambda^*(x_0)$ and $t\ge 0\mapsto \Lambda (t\lambda_0)$ is strictly convex at 1  then, with probability 1, for all $t\ge 1$ we have 
\begin{equation}\label{laplacebis}
\lim_{n\to\infty}  \Lambda_n^\omega(t\lambda_0)=\Lambda(\lambda_0)+(t-1)\scal{ \lambda_0}{ x_0}. 
\end{equation}

\end{enumerate}
\end{thm}
\begin{rem}{\rm 
(1) The almost sure large deviations equality \eqref{LD1} provided by Theorem~\ref{th3}(2) is a direct consequence of Theorem~\ref{th3}(1) and Theorem~\ref{LDPloc}.  

(2) In Theorem~\ref{th1}(3), since we consider a sequence of i.i.d.  real valued random variables, if $t\ge 0\mapsto \Lambda (t\lambda_0)$ is not strictly convex at 1 this means that $X$ is constant, and the result obviously still holds.
}\end{rem}

The following result is a direct consequence of Theorem~\ref{th3}(1) and~\ref{G-E}. 

\begin{cor}\label{corth1}
Assume {\bf (A1'')} and {\bf (A2-3)}. 

Suppose that $\sup_{x\in \mathcal{D}_\Lambda}\Lambda^*(x)<\infty$ and $\displaystyle \liminf_{n\to\infty} \frac{\log k(n)}{n}>\sup_{x\in \mathcal{D}_\Lambda}\Lambda^*(x)$, or $\displaystyle \lim_{n\to\infty} \frac{\log k(n)}{n}=\infty$. With probability 1, $\Lambda^\omega_n$ converges uniformly  to $\Lambda$ on the compact subsets of $\mathcal{D}_\Lambda$, hence the assertion of parts (1) and (2) of Theorem~\ref{G-E} hold for $(\mu_n^\omega)_{n\ge 1}$. If, moreover, $\Lambda$ is essentially smooth and lower semi-continuous,  the assertion of part (3) of Theorem~\ref{G-E} holds for $(\mu_n^\omega)_{n\ge 1}$, i.e. $(\mu_n^\omega)_{n\ge 1}$ satisfies in $\R^d$ the LDP with good rate function $I=\Lambda^*$. 
\end{cor}


Next we want to measure more finely how big must be $k(n)$ for $\Lambda_n^\omega(\lambda)$ to converge to $\Lambda(\lambda)$ when $\Lambda$ is smooth. 

If $\mathcal Y=\R^d$, for any $n\ge 1$ and any subset $B$ of $\mathcal{D}_\Lambda$ let 
\begin{eqnarray*}
\delta_{n}\Lambda(B)&=&\sup\{|\Lambda(\lambda)-\Lambda_n(\lambda)|:\lambda\in B\},\\ 
\delta_{n}\Phi(B)&=&(\sup_{\lambda\in B}\|\lambda\|) \|S_n\Phi-S_n\Phi_n\|_\infty/n,\\ 
\delta_n(\Lambda,\Phi)(B)&=&\delta_{n}\Lambda(B)+\delta_{n}\Phi(B),
\end{eqnarray*}
and if $\Lambda$ is twice continuously differentiable, let 
\begin{eqnarray*}
\varLambda^*(B)&=&\sup\{\Lambda^*(\grad\Lambda(\lambda)):\lambda\in B\},\\
\xi_1(B)&=&\sup\{\|\grad \Lambda(\lambda)\|: \lambda\in B\},\\
 \xi_2(B)&=&\sup\left \{\frac{1}{2}~^{t}\lambda {\rm D}^2\Lambda(\lambda)\lambda: \lambda\in B\right \}\\
 \xi(B)&=&\varLambda^*(B)+\xi_2(B),
\end{eqnarray*}
where ${\rm D}^2\Lambda(\lambda)$ stands for the Hessian matrix of $\Lambda$ at $\lambda$.

If $B\subset\R^d$ and $\rho\in\R_+^*$ we define $B_\rho$ as $\{\lambda\in\R^d: d(\lambda,B)\le \rho\}$, where ${\rm d}$ stands for the Euclidean distance. 

\begin{thm}\label{thspeed} Suppose that {\bf (A1')}, {\bf (A2')} and {\bf (A3)} hold, and $\Lambda$ is twice continuously differentiable over $\mathcal D$. Let $B$ be a compact subset of $\mathcal{D}$ and let $\rho>0$ such that $B_\rho\subset \mathcal D$. 

Suppose that there exists a positive sequence $(\epsilon_n)_{n\ge 1}$ converging to 0 such that
\begin{equation}\label{BC}
\sum_{n\ge 1} \exp \big (-\sqrt{\epsilon_n} [\log(k(n))-n\varLambda^*(B)]\big ) \epsilon_n^{-(d+3/2)}\exp \big ( 3 n[\xi(B_\rho)\epsilon_n+ \delta_n(\Lambda,\Phi)(B_\rho)]\big)<\infty.
\end{equation}
Let $\eta>0$. With probability 1, for $n$ large enough, 
\begin{equation}\label{unifconv}
\max_{\lambda\in B} |\Lambda^\omega_n(\lambda)-\Lambda(\lambda)|\le \mathcal E(n,\eta),
\end{equation}
where $\mathcal E(n,\eta)=(\eta+2\xi_1(B))\epsilon_n + \delta_n(\Lambda,\Phi)(B_\rho)$, or $\mathcal E(n,\eta)=\mathcal E(n)=\epsilon_n(1+\epsilon_n)/n+\delta_n(\Lambda,\Phi)(B_\rho)$ if $B$ consists of only one point.
\end{thm}

\begin{rem}\label{qual}
{\rm (1) It follows easily from the proof of Theorem~\ref{thspeed} (see \eqref{dyadique}) that if $B=\{\lambda\}$, then in \eqref{BC} one can replace $(d+3/2)$ by $3/2$ to get the same conclusions as in Theorem~\ref{thspeed}. 

\smallskip

\noindent (2) If the $X_i$ are i.i.d, a simple modification of the proof using Lemma~\ref{Rio}(2) rather than Lemma~\ref{Rio}(1) makes it possible to replace $(d+3/2)$ by $(d+1)$ in \eqref{BC}. 

\smallskip

\noindent (3) In the context described in Section~\ref{Birkhoff}, where $X$ takes values in a symbolic space,  $S_n\Phi(X)$ represents the Birkhoff sum of a continuous $\R^d$-valued potential $\Phi$ and  the law of $X$ is a Gibbs measure, we will give conditions under which both  $\delta_{n}\Lambda(B_\rho)$ and $\delta_{n}\Phi(B_\rho)$ are $O(1/n)$.  Then, a choice like $\epsilon_n=\gamma \log(n)/n$ and $\log(k(n))/n-\varLambda^*(B)\ge \sqrt{\gamma' \log(n)/n}$ with $\sqrt{\gamma\gamma'}> d+5/2+3\gamma \xi(B_\rho)$ ensures that \eqref{BC} holds and $\mathcal E(n,\eta)= O(\epsilon_n)=O(\log(n)/n)$.

%
}

\end{rem}

\begin{rem}{\rm 
As  a first explicit example of situation to which Theorems~\ref{th3} and~\ref{thspeed} can be applied, let us consider products of random invertible matrices applied to a normalized vector. Let $\mu$ be a probability measure on  $GL_m(\R)$.  Suppose that the support of $\mu$ generates a strongly irreducible and contracting semi-group (see Ch. III in \cite{BL} for the definition). Suppose also that $\exp(\tau \max (\log^+\|x\|,\log^+\|x^{-1}\|))$ is $\mu$-integrable for some $\tau>0$. Let $X=(X_i)_{i\ge 1}$ be a sequence of independent random matrices distributed according to $\mu$. Fix a unit vector $x$ and set $S_n\Phi(X)= \log \| X_n\cdots X_1\cdot x\|$. There exists (see Ch V.6 in \cite{BL}) a neiborhood $\mathcal D$ of $0$, independent of $x$, such that the limit $\Lambda$ of $\Lambda_n$ exists and is analytic on $\mathcal D$ (the derivative of $\Lambda$ at 0 is the upper Lyapounov exponent associated with $\mu$). 
}
\end{rem}
In the case where the $X_i$ take values in $\R^d$ and are i.i.d, we also have  the following improvement of Theorem~\ref{thspeed}.

\begin{thm}
\label{th2}
Suppose that the $X_i$ are i.i.d and take values in $\R^d$. Suppose also that $S_n\Phi(X)=\sum_{k=1}^nX_i$, and  $\Lambda(\lambda)=\log \E \exp( \scal{\lambda}{X_1})$ is finite over a convex open subset of $\R^d$. 

Let $B$ be a compact subset of $\mathcal{D}$ and let $\rho>0$ such that $B_\rho\subset \mathcal D$. Suppose that there exists a positive sequence $(\epsilon_n)_{n\ge 1}$ converging to 0  such that 
 $$
\sum_{n\ge 1} \exp \big (-\sqrt{\epsilon_n} [\log (k(n))-n\varLambda^*(B)]\big )\epsilon_n^{-(d+1)}\exp \big (\xi_2 (B_\rho)\epsilon_n n\big)<\infty.
$$
The same properties as in Theorem~\ref{thspeed} hold, with $\mathcal E(n,\eta)=(\eta+2\xi_1(B))\epsilon_n$, or  $\mathcal E(n,\eta)=\mathcal E(n)=\epsilon_n(1+\epsilon_n)/n$ if $B$ consists of only one point.
\end{thm}

\begin{rem}
If $B=\{\lambda\}$, in Theorem~\ref{th2} we can take $\epsilon_n=\gamma\log(n)/n$ and $\log(k(n))/n-\Lambda^*(\grad\Lambda(\lambda))\ge\sqrt{ \gamma' \log(n)/n}$ with $\sqrt{\gamma\gamma'}>d+2+\gamma\cdot \frac{1}{2}~^{t}\lambda {\rm D}^2\Lambda(\lambda)\lambda$. Then $\mathcal{E}(n)\le (1+\epsilon_n)\epsilon_n/n= \gamma\log(n)/n^{2} +\gamma^2\log(n)^2/n^3$.
\end{rem}
\medskip



\section{Examples}\label{examples}
This section describes various contexts to which our results can be applied. We investigate applications to Brownian motion (Section~\ref{state2}), dynamical systems and number theory (Sections~\ref{Birkhoff} and~\ref{Gauss}), branching random walks (Section~\ref{brw}) and Poissonian random walks  (Section~\ref{poisson}).

\subsection{Fluctuations of the increments of Brownian motion}\label{state2}

Let $(W_t)_{t\in [0,1]}$ be a $d$-dimensional standard Browian motion defined on a probability space $(\Omega,\mathcal A,\P)$.  Let $(k(n))_{n\ge 1}$ be a sequence of positive integers. For each $n\ge 1$ and $1\le j\le n$ we denote $[(j-1)/n,j/n]$ by $J_{n,j}$ and the increment of $W$ over the interval $J_{n,j}$ is then denoted by $\Delta W(J_{n,j})$. 

For every $\omega\in \Omega$ and $n\ge 1$, define 
$$
\mu^\omega_n=\frac{1}{k(n)} \sum_{j=1}^{k(n)} \delta_{x_{n,j}(\omega)},\text{ with }x_{n,j}(\omega)=(k(n)/n)^{1/2}\Delta W(J_{k(n),j}).
$$
In other words, for each Borel set $B\subset \R^d$, we have 
$$
\mu^\omega_n(B)=\frac{\displaystyle \#\Big \{1\le j\le k(n): (k(n)/n)^{1/2}\Delta W(J_{k(n),j})\in B\Big\}}{k(n)}.
$$

The following result is essentially a refinement of Theorem~\ref{th3} applied to a sequence of independent centered Gaussian vectors with covariance matrix the identity. We will give a short proof in Section~\ref{pfbrown1}. 
\begin{thm}\label{brown}
Let $R>0$. Suppose that there exists a positive sequence $(\epsilon_n)_{n\ge 1}$ converging to 0  such that 
\begin{equation}\label{condbrown}
\sum_{n\ge 1} \exp \big (-\sqrt{\epsilon_n} [\log (k(n))-n(1+\sqrt{\epsilon_n})R^2/2)]\big )\epsilon_n^{-1}<\infty.
\end{equation}
With probability 1, for every Borel subset $\Gamma$ of $\widering B(0,R)$,  \eqref{ldpgen} holds for $(\mu_n^\omega)_{n\ge 1}$, with rate  function $I(x)=\|x\|^2/2$. 
\end{thm}
The choice $\epsilon_n=\gamma\log(n)/n$ and $\log(k(n))/n-(1+\sqrt{\epsilon_n})R^2/2\ge\sqrt{ \gamma' \log(n)/n}$ with $\sqrt{\gamma\gamma'}>2$ yields \eqref{condbrown}.

We also have a functional result based  on the LDP established by Schilder (see \cite{DZ} Th. 5.2.3): for $n\ge 1$, let $\nu_n$ stand for the distribution of $W/\sqrt{n}$ as a random element of $C_0([0,1])$, the space of $\R^d$-valued continuous functions $\phi$ over $[0,1]$ such that $\phi(0)=0$. Then $(\nu_n)_{n\ge 1}$ is exponentially tight and satisfies in $C_0([0,1])$ the LDP with good rate function
$$
I(\phi)=\begin{cases}\frac{1}{2} \int_0^1\phi'(t)^2\,{\rm d}t& \text{if }\phi\in H^1\\
\infty&\text{otherwise}
\end{cases},
$$
where $H^1$ stands for Sobolev space of absolutely continuous elements of $C_0([0,1])$ with  square integrable derivative. 

It follows from Shilder's theorem that if $X=(X_i)_{i\ge 1}$ is a sequence of independent standard Brownian motions and $S_n=X_1+\cdots +X_n$, the distributions of the variables $S_n/n$ also satisfy in $C_0([0,1])$ the LDP with rate $I$. Consequently we get almost surely the LDP with rate $I$ for the local fluctuations of $S_n$ in the sense of Theorem~\ref{corASLDP}.  This essentially yields the following result.

For each $n\ge 1$ and $1\le j\le k(n)$ denote by $W_{k(n),j}$ the standard Brownian motion $t\in [0,1]\mapsto k(n)^{1/2}\big (W((t+(j-1))/k(n))-W((j-1)/k(n))\big )$.  
For every $\omega\in \Omega$ and $n\ge 1$, define 
$$
\mu^\omega_n=\frac{1}{k(n)} \sum_{j=1}^{k(n)} \delta_{x_{k(n),j}(\omega)},\text{ with }x_{k(n),j}(\omega)=\frac{W_{k(n),j}}{n^{1/2}}.
$$
\begin{thm}\label{brownfonc}
Suppose that $\displaystyle \lim_{n\to\infty} \frac{\log k(n)}{n}=\infty$. With probability 1, $(\mu_n^\omega)_{n\ge 1}$ satisfies in $C_0([0,1])$ the LDP with good rate function $I$. 
\end{thm}

\begin{rem}
It is possible to combine the ideas developed in this paper with those of \cite{Mog} to obtain results in the spirit of Theorem~\ref{brownfonc} for some L\'evy processes with jumps. 
\end{rem}

\subsection{Local fluctuations of Birkhoff sums and products of matrices with respect to Gibbs measures}\label{Birkhoff}

Let $\Sigma_m$ stand for the one sided symbolic space over a finite alphabet of cardinality $m\ge 2$: $\Sigma_m=\{0,\dots,m-1\}^{\N_+}$.  The set $\Sigma_m$ is endowed with the shift operation $T(\{t_n\}_{n=1}^\infty)= \{t_{n+1}\}_{n=1}^\infty$. Let $A$ be a $m\times m$ matrix with all entries equal to 0 and 1 and such that $A^p$ is positive for some $p\ge 1$. Then let $(\Sigma_A,T)$ be the associated topologically mixing subshift of finite type of $(\Sigma_m,T)$, i.e. $\Sigma_A=\{t\in\Sigma_m:\ \forall\ n\ge 1,\ A_{t_{n},t_{n+1}}=1\}$. 

We denote by $\mathcal M(\Sigma_A,T)$ the set of invariant probability measures under $T$. 

For $n\ge 1$ we define $\Sigma_{A,n}=\{(t_1\dots t_n)\in \{0,\dots,m-1\}^n: \ \forall\ 1\le k\le n-1,\ A_{t_{k},t_{k+1}}=1\}$. 

If $t\in \Sigma_A$  and $n\ge 1$ we denote $t_1\cdots t_n$ by $t_{|n}$ and for $w\in \Sigma_{A,n}$ the cylinder $\{t\in \Sigma_A:t_{|n}=w\}$ is denoted $[w]$. 

The set $\Sigma_A$ is also endowed with the standard ultra-metric distance $d(t,s)=m^{-|t\land s|}$, where $|t\land s|=\sup\{n:t_{|n}=s_{|n}\}$. 

If $\psi$ is a continuous function from  $\Sigma_A$ to $\R$, 
the topological pressure of $\psi$ is defined as $P(T,\psi)=\sup\{\nu(\psi)+h_\nu(T):\nu\in  \mathcal M(\Sigma_A,T)\}$, and one has (see \cite{Bowen})
$$
P(T,\psi)=\lim_{n\to\infty}\frac{1}{n}\sum_{w\in \Sigma_{A,n}}\sup_{y\in[w]} \exp (S_n\psi(y)).
$$

We say that $\psi$  satisfies the bounded distorsion property if
$$
\sup_{n\ge 1}v_n<\infty,\text{ where }v_n=\sup_{\substack{t,s\in \Sigma_A\\ t_{|n}=s_{|n}}} |S_n\psi(t)-S_n\psi(s)|<\infty.
$$
In this case, it is well known that $\sup\{\nu(\psi)+h_\nu(T):\nu\in  \mathcal M(\Sigma_A,T)\}$ is attained at a unique and ergodic measure called the equilibrium state of $\psi$ (see \cite{Bowen,Ru}). We will denote it by $\nu_\psi$.  This measure is a Gibbs measure, in the sense that there exists a constant $C>0$ such that 
\begin{equation}\label{gibbsm}
\forall\ n\ge 1,\ \forall t\in\Sigma_A,\ C^{-1}\exp(S_n\psi(t)-nP(T,\psi))\le \nu_\psi([t_{|n}])\le C\exp(S_n\psi(t)-nP(T,\psi)).
\end{equation}

Moreover, if $\Phi$ is a continuous mapping from $\Sigma_A$ to $\R^d$ such that each component of $\Phi$ satisfies the bounded distorsion property, then $\lambda\in\R^d\mapsto P(T,\langle \lambda,\Phi\rangle)$ is a $C^1$ mapping from $\R^d$ to $\R$ (see \cite{Ruelle2} and \cite{BF}).

\subsubsection{{\bf Results for Birkhoff sums}}\label{birksum} We fix a real valued potential $\psi$ on $\Sigma_A$  satisfying the bounded distorsion property.  Then, the process $X$ defined as the identity map of $\Sigma_A$ is stationary with respect to the ergodic measure $\nu_\psi$. We also fix $\Phi$, a continuous mapping from $\Sigma_A$ to $\R^d$ and define $(S_n\Phi(X))_{n\ge 1}$ as the sequence of Birkhoff sums of $\Phi$.
 
Thus, setting $(\Omega,\P)=(\Sigma_A,\nu_\psi)$, the quantities introduced in Section~\ref{State} take the following form. For all $n\ge 1$, $B\in B_{\R^d}$ and $\lambda\in\R^d$,
$$
\mu_n(B)=\nu_\psi \big (\{t\in \Sigma_A: \ S_n\Phi(t)/n\in B\}\big )
$$  
and 
$$
\Lambda_n (\lambda) =  \frac{1}{n} \log \E\exp\big ( \scal{\lambda}{S_n \Phi(X)} \big ) 
= \frac{1}{n}\int_{\Sigma_A}\exp( \scal{\lambda}{S_n\Phi(t)})\, {\rm d}\nu_\psi(t).
$$ 
Also, $\mu^\omega_n$ and $\Lambda_n^\omega$ are denoted $\mu^t_n$ and $\Lambda_n^t$ respectively and we have for $t\in \Sigma_A$,  $n\ge 1$,  and $\lambda\in\R^d$
$$
\mu^t_n(B)=\frac{\displaystyle \#\big \{1\le j\le k(n): S_n \Phi (T^{(j-1)n}t)/n\in B\big\}}{k(n)}.
$$
and
$$
\Lambda^t_n(\lambda)=\frac{1}{n}\log \int_{\mathbb{R}^d}\exp(n\scal{\lambda}{x}){\rm d}\mu_n^t (x).
$$

Due to the Gibbs properties of $\nu_\psi$ \eqref{gibbsm}, $\Lambda(\lambda)=\lim_{n\to\infty}\Lambda_n (\lambda)$ exists and takes the form
$$
\Lambda(\lambda)=P(T,\psi+ \langle \lambda,\Phi\rangle)-P(T,\psi).
$$ 
If, moreover, each component of $\Phi$ satisfies the bounded distorsion property then $\Lambda$ is $C^1$. Thus, condition {\bf(A1'')} (hence {\bf (A1')}) hold with $\mathcal{D}_\Lambda=\R^d$. Moreover, $\delta_n\Lambda(B)=O(1/n)$ for bounded sets $B$. 

For {\bf (A2)} to hold we must ask some mixing properties of $\nu_\psi$. It is quite simple to see that {\bf (A2)} holds under the stronger assumption that there exists $\gamma>0$ and $\theta>1$ such that $\alpha_{X,m}=O(\exp(-\gamma \log(m)^{\theta}))$.  Then, due to Theorem 1.11 in \cite{Baladi},  {\bf (A2)} holds as soon as the modulus of continuity of $\psi$, namely $\kappa(\psi,\cdot)$ satisfies $\kappa(\psi,\delta)=O\big (\exp\big (-\gamma(\log|\log(\delta)|)|^\theta\big )\big )$ as $\delta\to 0$ for some $\gamma>0$ and $\theta>1$. Also  {\bf (A2')} holds as soon as $\kappa(\psi,\delta)=O\big (\exp\big (-\gamma|\log(\delta)|^\theta\big )\big )$ as $\delta\to 0$ for some $\gamma>0$ and $\theta>0$.

The function $\Phi$ being continuous on the compact set $(\Sigma_A,d)$, {\bf (A3)} always holds since we can always approximate $\Phi$ by a function $\Phi_n$ depending only on $(t_1,\dots,t_n)$ so that $\|S_n\Phi-S_n\Phi_n\|_\infty\le \sum_{k=1}^n\kappa(\Phi, m^{-k})=o(n)$.  

\medskip

Thus under the above conditions on $\Phi$ and $\psi$ assuring {\bf(A1'')} and {\bf(A2)} Theorem~\ref{th3} and Corollary~\ref{corth1} can be applied to this context and provide information regarding the convergence of $\Lambda_n^t$ to $\Lambda$ for $\nu_\psi$-almost every $t$. If, moreover, we assume that $\psi$ and the components of $\Phi$ are H\"older continuous, then $\Lambda$ is analytic (see for instance Th. 5 in \cite{Ru}) and {\bf(A2')} holds, so that we can apply Theorem~\ref{thspeed}. 

\medskip

In fact, even if $\Phi$ is only supposed continuous, $(\mu_n)_{n\ge 1}$ satisfies  in $\R^d$ the LDP with good rate function 
\begin{equation}\label{lambda*gibbs}
I(x)=\begin{cases}  \inf\Big\{P(T,\psi)- (h_\nu(T)+\nu(\psi)):\nu\in \mathcal{M}(\Sigma_A,T),\ \nu(\Phi)=x\Big\}=\Lambda^*(x)&\text{ if }x\in \mathcal{D}_I\\
\infty&\text{ otherwise}
\end{cases},
\end{equation}
where $\mathcal{D}_I=\{\nu(\Phi):\nu\in\mathcal{M}(\Sigma_A,T)\}$, and $I$ is bounded over the compact convex set $\mathcal{D}_I$. This LDP essentially follows from Theorem 6 of \cite{Young} (which deals with H\"older potentials), and the duality between the pressure and entropy functions (see \cite{FF,FFW,TV,FH} for details and related works). Thus {\bf (A1)} holds. It follows that we can apply  Theorem~\ref{corASLDP} and transfer the previous LDP to the local fluctuations  of $S_n\Phi$: 
\begin{thm}\label{thbirkhoff}
If $\displaystyle \liminf_{n\to\infty} \frac{\log k(n)}{n}>\sup_{x\in {\mathcal{D}_I}}\Lambda^*(x)$, then for $\nu_\psi$-almost every $t$, the sequence $(\mu^t_n)_{n\ge 1}$ satisfies in $\R^d$ the LDP with good rate function given by \eqref{lambda*gibbs}.
\end{thm}
Thus, we can also deal with the cases where the function $\Lambda$ is non differentiable at some $\lambda\in\R^d$ because $\scal{\lambda}{\Phi}+\psi$  have at least two equilibrium states with distinct entropies (see \cite{Ru} p. 52 for instance).

\medskip

\noindent{\it Some geometric applications.} The previous results have applications to geometric realizations of $(\Sigma_A,T)$, for instance on repellers of topologically mixing $C^{1+\epsilon}$ conformal maps  of Riemannian manifolds. For such a repeller $(J,f)$, they make it possible to describe the local fluctuations of $S_n \log \|Df\|$ almost everywhere with respect to any enough mixing Gibbs measure on $(J,f)$; this means that while with respect to such a measure $\nu_\psi$ one observes on almost every orbit an expansion ruled by a fixed Lyapounov exponent equal to $\nu_\psi(\log\|Df(x)\|)$, we can finely quantify local fluctuations with respect to this global property. The same can be done along the stable and unstable manifolds on locally maximal invariant sets of topologically mixing  Axiom A diffeomorphisms (see \cite{Bowen,Katok} for details on these dynamical systems). 

\medskip

Another application concerns the harmonic measure on planar Cantor repellers of $C^{1+\epsilon}$ conformal maps $f$; recall that given such a repeller $J$, this measure is the probability measure $\mu$ such that for each
 $t\in J$ and $r>0$, $\mu(B(t,r))$ is the probability that a planar Brownian
  motion started at $\infty$ attains $J$ for the first time at a point of $B(t,r)$. It turns out that $\mu$ is equivalent to the equilibrium state $\mu_\varphi$ of a H\"older potential $\varphi$ on $J$ (see \cite{Ca} or \cite{Ma}). Given another  enough mixing  Gibbs measure $\nu_\psi$, the ergodic theorem ensures that $\lim_{r\to 0^+}\log\mu_\varphi(B(t,r))/\log (r)=(P(\varphi)-\nu_\psi (\varphi))/\nu_\psi(\log\|Df\|)$ for $\nu_\psi$-almost every $t$. Then, our result yields information on the fluctuations with respect to this behavior.  Indeed, one can use the coding of $(J,f)$ by a subshift of finite type thanks to a Markov partition and apply our results to the pair $(S_n(\varphi-P(\varphi)),S_n\log\|Df\|)$. If we remember the origin of $\varphi$, this yields information on the local fluctuations of the Brownian motion around $\nu_\psi$-almost every $t$. This can be made more explicit in the case that $J$ is self-similar and homogeneous, for instance when $J=K^2$ with $K$ the middle third Cantor set. There our results provide, for $\nu_\psi$-almost every $t$, information on the distributions of the values 
  $$
 \frac{ \log \mu_\varphi(C_{jn}(t))}{\log \mu_\varphi(C_{(j-1)n}(t))}=1+\frac{S_n(\varphi-P(\varphi))(f^{(j-1)n}(t))}{jn\nu_\psi(\varphi-P(\varphi))}+\frac{\epsilon(jn)}{j}, \ 1\le j\le k(n),$$ 
 where $C_k(t)$ is the triadic cube of generation $k$ containing $t$ and $\lim_{k\to\infty}\epsilon(k)=0$. 
 
 \medskip
 
 The previous interpretations of our results about  the local behavior of Gibbs measures can be extended to the case of Axiom A diffeomorphisms invoked above. 
 
 Thus, to summarize, while \cite{OP,OP2,Young} provide large deviations with respect to the almost sure asymptotic behavior of Birkhoff sums on a hyperbolic invariant set endowed with a Gibbs measure, our results provide a natural complement by describing the fluctuations with respect to this behavior on almost every orbit viewed by this measure.

\medskip

The next two subsections briefly discuss extensions to norms of Birkhoff products of matrices of the previous properties of Birkhoff sums of potentials. 

\subsubsection{\bf Birkhoff products of positive matrices} Suppose that $M$ is a mapping from $\Sigma_A$ to the set of positive square matrices of order $d\ge 1$, and fix an enough mixing Gibbs measure $\nu_\psi$. If the components $M_{i,j}$ are so that $\log(M_{i,j})$ has the bounded distorsion property then, one can apply Theorem~\ref{th3} to $S_n\Phi(X(t))=S_n\Phi(t)=\log \|M(t)M(T(t))\cdots M(T^{n-1}(t))\|$ with respect to $\nu_\psi$.  Indeed, the convergence of $\Lambda_n (\lambda)$ for $\lambda\in \R$ comes from the Gibbs property of $\nu_\psi$ and the subadditivity and superadditivity properties of  $S_n\Phi(X)$, and the differentiability of $\Lambda$ comes from the variational principle for subadditive potentials (see \cite{Feng04} for instance). If the components of $M$ are only supposed continuous $\Lambda_n (\cdot)$ still converges but may be non differentiable. It is then possible to extend the result explained in the previous subsection and show that {\bf (A1)} holds for $S_n\Phi$ with the good rate function $I$ still satisfying \eqref{lambda*gibbs} and bounded over $\mathcal{D}_I$. The only difference is that here $\nu(\Phi)$ is defined as $\lim_{n\to\infty}n^{-1} \int_{\Sigma_A}S_n\Phi(t)\, {\rm d}t$.

\subsubsection{\bf Bernoulli products of invertible matrices} Suppose that we are given $M_1,\dots,M_m$, $m$ matrices of $GL(d,\mathbb C)$  such that there is no proper non-zero linear subspace $V$ of $\mathbb C^d$ such that $M_i(V)\subset V$. Then, define $M(t)=M_{t_1}$ and $S_n\Phi(X(t))=S_n\Phi(t)=\log \|M_{t_1}\cdots M_{t_n}\|$ for $t\in \Sigma_m$. Nice superadditivity and subadditivity properties (see \cite{Feng2009}) make it possible to extend the results of the previous section to this context.  We do not enter into the details.

\subsection{Local fluctuations in the continued fraction expansion of Lebesgue-almost every point}\label{Gauss}
The interval $[0,1)$ is endowed with the dynamics of the Gauss transformation $f(0)=0$, $f(t)=1/t-\lfloor 1/t\rfloor$ if $t\in (0,1)$. Then, the continued fraction expansion of an irrational number $t\in (0,1)$ is represented by the sequence $[a_1(t);a_2(t);\dots; a_n(t);\dots]$, where $a_1(t)=\lfloor 1/t\rfloor $ and $a_n(t)=a_1(f^{n-1}(t))=\lfloor 1/f^{n-1}(t)\rfloor $. The Gauss measure $\mu_G$ whose density with respect to the Lebesgue measure on $[0,1)$ is  $1/(1+t)\log(2)$ is ergodic with respect to $f$, and it possesses the strong mixing properties required in {\bf (A2)} (see \cite{Bil} for instance). Now let $\Phi(t)=\log a_1(t)$ for $t\in (0,1)$. An application of the Birkhoff ergodic theorem proves that for Lebesgue almost every  $t$, one has $S_n\Phi(t)=\sum_{k=0}^{n-1}\log a_k(t)\sim n \int_0^1\log a_1(t)\ {\rm d}\mu_G(t)$. 

Here we are concerned with the limit of  $\Lambda_n(\lambda)=n^{-1}\log\int_0^1\exp(\lambda S_n\Phi(t)) \ {\rm d}\mu_G(t)$ whenever it exists. For each $n\ge 1$ and each sequence $a_1,\dots a_n$ of integers let us denote by $I_{a_1,\cdots, a_n}$ the interval $\{t\in [0,1): [a_1(t);a_2(t);\dots; a_n(t)]=[a_1;a_2;\dots; a_n]\}$. It is clear that the question reduces to studying $n^{-1}\log \sum_{(a_1,\dots,a_n)\in (\N_+)^n} (a_1\cdots a_n)^\lambda|I_{a_1,\cdots, a_n}|$; this sequence converges for $\lambda<1$ to a limit $\Lambda(\lambda)$ analytic in $\lambda$ (see Section 4 of \cite{FanKhintchine}). Consequently, Theorem~\ref{th3}, Corollary~\ref{corth1} and Theorem~\ref{thspeed} provide large deviations properties for the local fluctuations of $\log (a_1(t))+\cdots+\log (a_n(t))$ almost everywhere with respect to the Lebesgue measure. 

\medskip

The previous example can be generalized by studying the local fluctuations of the Birkhoff sums associated with good potentials on the symbolic space over an infinite alphabet with respect to enough mixing Gibbs measures. We refer the reader to \cite{FanKhintchine} and \cite{Sarig} for further examples and references.

\subsection{Local fluctuations of branching random walks (BRW) with respect to generalized branching measures}\label{brw}$\ $

Let $(N,(\psi_1,\Phi_1),(\psi_2,\Phi_2),\dots)$ be a random vector taking values in $\N_+\times ({\R}\times \R^d)^{\N_+}$. In the sequel, the distribution of $N$ will define a supercritical Galton-Watson tree, on the boundary of which will live a Mandelbrot measure determined by $(\psi_1,\psi_2,\dots)$, with respect to which we will look almost everywhere at the local fluctuations of a branching random walk whose distribution is determined by $(\Phi_1,\Phi_2,\dots)$. Here, $(\psi_1,\psi_2,\dots)$ and $(\Phi_1,\Phi_2,\dots)$ play roles analogous to the potentials $\psi$ and $\Phi$ in the previous section. 

Let $\{(N_{u0},(\psi_{u1},\Phi_{u1}),(\psi_{u2},\Phi_{u2}),\dots)\}_u$ be a family of independent copies of the vector $(N,(\psi_1,\Phi_1),(\psi_2,\Phi_2),\dots)$ indexed by the finite sequences $u=u_1\cdots u_n$, $n\ge 0$, $u_i\in\N_+$ ($n=0$ corresponds to the empty sequence denoted $\emptyset$), and let $\TT$ be the Galton-Watson tree with defining elements $\{N_u\}$: we have $\emptyset \in \TT$ and, if $u\in \TT$ and $i\in\N_+$ then $ui$, the concatenation of $u$ and $i$, belongs to $\TT$ if and only if $1\le i\le N_u$. Similarly, for each $u\in \bigcup_{n\ge 0} \N_+^n$, denote by $\TT(u)$ the Galton-Watson tree rooted at $u$ and defined by the $\{N_{uv}\}$, $v\in \bigcup_{n\ge 0} \N_+^n$.

The probability space over which these random variables are built is denoted $(\Upsilon,\mathcal A, \PP)$, and the expectation with respect to $\PP$ is denoted $\EE$. 

\medskip

Let us define the $\R\cup\{\infty\}$-valued convex mapping
$$
\LL:\lambda\in\R^d\mapsto \log  \EE\Big (\sum_{i=1}^N\exp(\psi_i+\scal{\lambda}{\Phi_i}) \Big ).
$$
We assume that 
$$
\EE\Big (\sum_{i=1}^N\exp(\psi_i) \Big )=1,\ \log \EE\Big (\sum_{i=1}^N\psi_i\exp(\psi_i) \Big )<0\text{ and } \EE\Big (\Big (\sum_{i=1}^N\exp(\psi_i) \Big )\log^+\Big (\sum_{i=1}^N\exp(\psi_i) \Big )\Big )
<\infty.$$
Then, it is known (see \cite{MJFM,KP,LPP}) that for each $u\in \bigcup_{n\ge 0} \N_+^n$, the sequence
$$
Y_n(u)=\sum_{v=v_1\cdots v_n\in \TT(u)}\exp(\psi_{uv_1}+\cdots +\psi_{uv_1\cdots v_n})
$$
is a positive uniformly integrable martingale of expectation 1 with respect to the natural filtration. We denote by $Y(u)$ its $\PP$-almost sure limit. By construction, the random variables so obtained are identically distributed and positive. Also, the Galton-Watson tree $\TT$ is supercritical. 

Now, for each $u\in  \bigcup_{n\ge 0} \N_+^n$, we denote by $[u]$ the cylinder $u\cdot{\N_+}^{\N_+}$ and define
$$
\nu([u])=\mathbf{1}_{T}(u)\exp(\psi_{u_1}+\cdots +\psi_{u_1\cdots u_n})\, Y(u).
$$
Due to the branching property $Y(u)=\sum_{i=1}^{N_u}\exp(\psi_{ui})Y(ui)$, this yields a non-negative additive function of the cylinders, so it can be extended into a random measure $\nu_\gamma$ ($\gamma\in\Upsilon)$ on $\N_+^{\N_+}$  endowed with the Borel $\sigma$-field $\mathcal B=\mathcal B(\N_+^{\N_+})$. This measure has $\partial \TT=\bigcap_{n\ge 0}\bigcup_{u=u_1\cdots u_n\in T}[u]$ as support. 

Now, let $\Omega=\Upsilon\times \N_+^{\N_+}$. We can define on $(\Omega, \mathcal A\otimes \mathcal B)$ the probability measure
$$
\P(A)=\int_{\Upsilon}\int_{{\N_+}^{\N_+}}\mathbf{1}_A(\gamma,t){\rm d}\nu_\gamma(t) {\rm d}\PP(z).
$$ 
Then, it is known (see~\cite{LiuRouault} for instance) that the random variables $X_n(\gamma,t)=\Phi_{t_1\cdots t_n}(\gamma)$ are i.i.d. with respect to $\P$. If, moreover, $\grad \LL(0)$ exists then it equals $\E( {X}_1)$ and $({X}_1+\cdots  {X}_n)/n$ tends to $\grad \LL(0)$ $\P$-almost surely. In terms of the BRW $\sum_{i=1}^n\Phi_{t_1\cdots t_i}$ on $\TT$, this means that with $\PP$-probability 1, for $\nu_\gamma$-almost every $t\in \partial \TT$, we have $\lim_{n\to\infty}\sum_{i=1}^n\Phi_{t_1\cdots t_i}(\gamma)/n=\grad \LL(0)$. 

Moreover, in the present context, if we set $X=({X}_i)_{i\ge 1}$, since the $X_i$ are i.i.d. we have $\Lambda(\lambda)=\log \E\exp(\scal{\lambda}{X_1})=\LL(\lambda)$. Consequently, if $\LL$ is finite on an open convex subset $\mathcal{D}$ of $\R^d$, local fluctuations of the BRW $ \sum_{i=1}^n\Phi_{t_1\cdots t_i}$ are described $\PP$-almost surely $\nu_\gamma$-almost everywhere thanks to Theorem~\ref{th3}, Corollary~\ref{corth1} and Theorem~\ref{thspeed}. When $\Phi_i\in\{0,1\}$ for all $i\ge 1$, this is related to percolation on the Galton-Watson tree $\TT$ (see \cite{Lyons}).

\subsection{Local fluctuations of Poissonian random walks and covering numbers with respect to compound Poisson cascades}\label{poisson}$\ $

As in the previous section, the probability space over which we are going to define random variables is denoted $(\Upsilon,\mathcal A, \PP)$, and the expectation with respect to $\PP$ is denoted $\EE$. 
 
 \medskip
 
 Let $\xi>0$ and $\mathcal P$ a Poisson
point process in $\mathbb{R}\times (0,1]$ with intensity $\Lambda$
given~by 
$$
\displaystyle \Lambda ({\rm d}s
{\rm d}\lambda)=\frac{\xi {\rm d}s
{\rm d}\lambda}{\lambda^2}.
$$ 
For every $(s,\lambda)\in\mathcal P$ let $J(s,\lambda)=(s,s+\lambda)$. The question of knowing whether $\R_+\setminus \{0\}$ is or not  almost surely covered by the intervals $J(s,\lambda)$ has been raised in \cite{Mcov} in connexion with a similar problem previously raised in \cite{Dvo} for random arcs on the circle. These problems have been solved in \cite{S1,S2} (see also \cite{Kahane} for further information on this question). Then, works \cite{Fan,BarralFan} have been dedicated to the geometric heterogeneity of the asymptotic behavior of the covering numbers defined as follows (in fact, all the works mentioned above consider more generally the case of Poisson intensities invariant by horizontal translation). Here we rather look at local fluctuations of these numbers. 

\medskip


\smallskip

For every $t\in [0,1]$ and $n\ge 0$, the covering number of $t$
at height $e^{-n}$ by the Poissonian intervals $J(s,\lambda)$
 is defined as
\begin{eqnarray*}
N_n (\gamma,t)=\sum_{(s,\lambda)\in \mathcal P, \ \lambda>
e^{-n}}\mathbf{1}_{\{J(s,\lambda)\}}(t)=\#\big
\{(s,\lambda)\in \mathcal P:\ \lambda>e^{-n} , \ t\in J(s,\lambda)\big
\}\quad (\gamma\in\Upsilon).
\end{eqnarray*}
For every $t\in [0,1]$, this covering number can be seen as the ``Poissonian'' random walk $S_n(\gamma,t)=X_1(\gamma,t)+\cdots+X_n(\gamma,t)$ associated with the random variables $X_i(\gamma,t)$ defined~as
\begin{eqnarray*}
X_i(\gamma,t)&=&\sum_{\substack{(s,\lambda)\in \mathcal P, \\ e^{-i}<\lambda\le e^{-(i-1)}}}\mathbf{1}_{\{J(s,\lambda)\}}(t)\\
&=&\#\big
\{(s,\lambda)\in \mathcal P:\ e^{-i}<\lambda\le e^{-(i-1)}, \ t\in J(s,\lambda)\big
\} \quad (i\ge 1).
\end{eqnarray*}
The choice of $\Lambda$ ensures that the $X_i(\cdot, t)$ are i.i.d. We can describe the fluctuations of $S_n(\gamma,t)$ thanks to Theorems~\ref{th3} and~\ref{th2} by considering random measures on $\R_+$, namely compound Poisson cascades \cite{BM}. In fact, the invariance by horizontal translation of the constructions makes it possible to restrict ourselves to $[0,1]$ without loss of generality.

It turns out that we can also describe a more general model of Poissonian random walks in the spirit of branching random walks. To to this, we consider a random vector $(\psi,\Phi)\in\R\times \R^d$, and to each $(s,\lambda)\in \mathcal{P}$ we associated a copy $(\psi_{(s,\lambda)},\Phi_{(s,\lambda)})$ of  $(\psi,\Phi)$ in such a way that these random variables are independent and independent of $\mathcal{P}$.

For each $t\in[0,1]$ and $\phi\in\{\psi,\Phi\}$ we consider the random variables 
$$
X^\phi_i(\gamma,t)=\sum_{\substack{(s,\lambda)\in \mathcal P, \\ e^{-i}<\lambda\le e^{-(i-1)}}} \phi_{(s,\lambda)} \quad (i\ge 1,\ \gamma\in\Upsilon)
$$
as well as the Poissonian random walk $
S^\phi_n(\gamma,t)=X^\phi_1(\gamma,t)+\cdots+X^\phi_n(\gamma,t)$. An easy calculation shows that for any $(q,\lambda)\in\R\times \R^d$, for every $t\in\R_+$  one has $\EE\exp \big (qS_n^\psi(\cdot, t)+\scal{\lambda}{S_n^\Phi(\cdot,t)}\big )=\exp \big (n\xi\EE\big(\exp(q\psi+\scal{\lambda}{\Phi})-1\big)\big ).$

\medskip

We define over $[0,1]$ the sequence of random measures introduced in \cite{BM} as
\begin{equation}
\label{cpc}
\nu_{\gamma, n}({\rm d}t)= \big (\EE\exp(S_n^\psi (\cdot,t))\big )^{-1}\exp(S_n^\psi (\gamma,t))\,
{\rm d}t =  \exp\big (S_n^\psi (\gamma,t)-n\xi\EE\big(\exp(\psi)-1\big)\big )\big)\, {\rm d}t.
\end{equation}

Let $\tau(q)= (1-\xi)(1-q)+\xi\EE\big (\exp(q\psi)-q\exp(\psi)\big )$. We assume that $\tau'(1)<0$. Then, for $\PP$-almost every $\gamma$, $\nu_{\gamma,n}$ converges in the weak-star topology to a fully supported measure $\nu_\gamma$ over $[0,1]$, and whose total mass has expectation 1 (see \cite{BM}). We can defined on $\Omega=\Upsilon\times [0,1]$ endowed with $\mathcal A\otimes \mathcal B([0,1])$ the probability measure 
$$
\P(A)=\int_{\Upsilon}\int_{[0,1]}\mathbf{1}_A(\gamma,t){\rm d}\nu_\gamma(t) {\rm d}\PP(\gamma).
$$ 
Let 
$$
L:\lambda\in\mathbb{R}^d\mapsto \xi\EE\big(\exp(\psi+\scal{\lambda}{\Phi})-1\big ).
$$
The random variables $X^\phi_i(\gamma,t)$ are i.i.d with respect to $\P$, and it is not difficult to see that $\Lambda(\lambda)=\LL(\lambda)$. Thus, if $\LL$ is finite on an open convex subset $\mathcal{D}$ of $\R^d$, Theorems~\ref{th3}, Corollary~\ref{corth1} and Theorem~\ref{th2} applied to $X=(X_i)_{i\ge 1}$ with respect to $\P$ provide a description of the local fluctuations of $S^\Phi_n(\gamma,t)$, $\PP$-almost surely, for $\nu_\gamma$-almost every $t$.
\section{Conjecture on the ``randomness'' of fundamental constants}\label{Pi}
As mentioned in the introduction, our results lead us to formulate a new conjecture regarding how in any integer basis $m$ the digits of fundamental constants such as the number Pi or the Euler constant  look like almost every realization of a sequence of i.i.d random variables uniformly distributed in $\{ 0, \cdots, m-1\}$. This conjecture implies the normality property.

Recall the notations of Section~\ref{Birkhoff}. Consider a $\R^d$-valued continuous potential $\Phi$ defined on $\Sigma_m=\{0,\dots,m-1\}^{\N_+}$ endowed with the shift operation denoted $T$.  Consider a sequence $(k(n))_{n\ge 1}$ of positive integers. Recall that in Section~\ref{birksum} we have defined  for $t\in \Sigma_m$ the sequence of Borel measures $(\mu^t_n)_{n\ge 1}$ and logarithmic generating functions $(\Lambda^t_n)_{n\ge 1}$ as
$$
\mu^t_n=\frac{1}{k(n)} \sum_{j=1}^{k(n)} \delta_{x_{n,j}(t)}\text{ with } x_{n,j}(t)=S_n\Phi(T^{(j-1)n}t)/n
$$
and 
\begin{equation}\label{lambdaemp}
\Lambda^t_n(\lambda)=\frac{1}{n}\log \int_{\mathbb{R}^d}\exp(n\scal{\lambda}{x}){\rm d}\mu_n^t (x).
\end{equation}
Consider now the potential $\psi=0$ and the associated equilibrium state $\nu_\psi$,  i.e.  the measure of maximal entropy on $(\Sigma_m,T)$. We have $P(\psi)=\log(m)$.  

The process $X=(X_i)_{i\ge 1}$ defined on the probability space $(\Sigma_m,\nu_\psi)$ as $X(t)=(t_i)_{i\ge 1}$ is a sequence of  i.i.d random variables uniformly distributed in $\{ 0, \cdots, m-1\}$, and the rate function $I$ provided by \eqref{lambda*gibbs} takes the form
$$
I(x)=\begin{cases}  \inf\Big\{\log (m)-h_\nu(T):\nu\in \mathcal{M}(\Sigma_A,T),\ \nu(\Phi)=x\Big\}=\Lambda^*(x)&\text{ if }x\in \mathcal{D}_I\\
\infty&\text{ otherwise}
\end{cases},
$$
where $\mathcal{D}_I=\{\nu(\Phi):\nu\in\mathcal{M}(\Sigma_A,T)\}$ and $\Lambda(\lambda)=P(\scal{\lambda}{\Phi})$ for $\lambda\in\R^d$. 

\medskip

Let $D=(D_i)_{i\ge 1}$ be a sequence of digits in the integer basis $m$. We say that $D$ satisfies property $(\mathcal P)$ if 

\medskip

\noindent
{\bf Property $(\mathcal P)$:} The sequence $(\mu^D_n)_{n\ge 1}$ obeys in $\R^d$ the same LDP with rate $I$ as that provided by Theorem~\ref{thbirkhoff} for $(\mu_n^t)_{n\ge 1}$ (for $\nu_\psi$-almost every $t$).  and 

%
%
%
\begin{thm}
Property $(\mathcal P)$ implies the normality of $\sum_{i\ge 1}D_im^{-i}$ in basis~$m$.
\end{thm}

\begin{proof}
We prove the equivalent following fact: Property $(\mathcal P)$ implies that for all  real-valued continuous function $\varphi$ on $\Sigma_m$, $\lim_{p\to\infty} S_p\varphi (D)/p=\nu_\psi(\varphi)$. 

Let $\varphi$ be a real-valued continuous function on $\Sigma_m$. For $n\ge 1$ let $k_1(n)=(2m)^n$. We have $(n+1)k_1(n+1)\le 4m  nk_1(n)$ for all $n\ge 1$. 
Fix an integer $n_0\ge 1$, and for $n\ge n_0$,  $1\le i \le 4m-1$ and $0\le \ell \le m^{n_0}$ let $k_{i,\ell} (n)= (i+\ell m^{-n_0}) k_1(n)$.  The LDP of  property $(\mathcal P)$ holds for every sequence $(k(n))_{n\ge n_0}$ with $k(n)\in\{k_{i,\ell}(n):1\le i\le 4m-1,\ 0\le \ell\le m^{n_0}\}$, since  $\liminf_{n\to \infty} \log(k(n))/n>\log(m)$. 

If $p$ is a positive integer larger than $k_1(1)$, let $n_p$ be the largest integer $n$ such that $nk_1(n)\le p$.  By construction, $p\le 4m   n_pk_1(n_p)$. Let $(i_p,\ell_p)$ be the unique pair in $\{(i,\ell):1\le i\le 4m-1,\ 0\le \ell\le m^{n_0}\}$ such that $n_pk_{i_p,\ell_p}(n_p)\le p<n_pk_{i_p,\ell_p}(n_p)+ n_pk_{1}(n_p) m^{-n_0}$.

We have 
$$
\frac{S_p\varphi(D)}{p}=\frac{S_{n_pk_{i_p,\ell_p}(n_p)}}{n_pk_{i_p,\ell_p}(n_p)}+O(m^{-n_0})\quad (\text{as $p\to\infty$}),
$$
where the constant in $O(m^{-n_0})$ depends only on $\varphi$. Consequently, since at fixed $n_0$ we deal with the finite number of sequences $k(n)\in \{k_{i,\ell}(n):1\le i\le 4m-1,\ 0\le \ell\le m^{n_0}\}$, if we prove that the LDP of property $(\mathcal P)$ implies that $\lim_{n\to\infty} S_{nk(n)}\varphi (D)/nk(n)=\nu_\psi(\varphi)$ for each such sequence, we will get $\limsup_{p\to\infty} |\frac{S_p\varphi(D)}{p}-\nu_\psi(\varphi)|=O(m^{-n_0})$. Then, letting $n_0$ tend to $\infty$ will yield the desired conclusion. 

We reduced the problem to showing that $\lim_{n\to\infty} S_{nk(n)}\varphi (D)/nk(n)=\nu_\psi(\varphi)$ whenever $\liminf_{n\to \infty} \log(k(n))/n>\log(m)$. Suppose that $\liminf_{n\to \infty} \log(k(n))/n>\log(m)$. For $\epsilon>0$, we can write
$$
\left |\frac{S_{nk(n)}\varphi (D)}{nk(n)}-\nu_\psi(\varphi)\right |\le \int_{\R}|x-\nu_\psi(\varphi)|\,{\rm d}\mu_n^D(x)\le \epsilon+2\|\varphi\|_\infty\mu_n^D(\{x:|x-\nu_\psi(\varphi)|>\epsilon\}),
$$
and due to property $(\mathcal P)$, $\mu_n^D(\{x:|x-\nu_\psi(\varphi)|>\epsilon\})$ tends to $0$ as $n\to\infty$. Consequently, $
\limsup_{n\to\infty}\left |\frac{S_{nk(n)}\varphi (D)}{nk(n)}-\nu_\psi(\varphi)\right |\le \epsilon$ for all $\epsilon>0$.
\end{proof} 

\begin{rem}
{\rm One can wonder if, conversely, the normality of $\sum_{i\ge 1}D_im^{-i}$ implies property $(\mathcal P)$ for $D$. To begin with this question, it is interesting to seek an explicit normal number in basis $m$ for which property $(\mathcal P)$ holds; Champernowne's constant $C_m$ should be investigated.

}\end{rem}

Our conjecture is the following. 
\begin{con}\label{conjec} For every integer $m\ge 2$, the digits of the fractional part of either Pi or the Euler constant in basis $m$ satisfy $(\mathcal P)$. 
\end{con}

Conjecture~\ref{conjec} is supported by numerical experiments, which focus on the validity of the conclusions of Theorem~\ref{th3} for $(\Lambda^D_n)_{n\ge 1}$. From the numerical point of view, the most tractable situations concern potentials that are constant over the cylinders of the first generation. In the context of digit frequency associated to normality of numbers, it is natural to consider potentials of the form $\Phi_a(t)=\mathbf{1}_{\{a\}}(t_1)$, with $a\in \{0,\dots, m-1\}$. Here, we show simulation results when $m=10$ and $a=0$; in this case $\displaystyle \Lambda (\lambda)= \log \frac{9+\exp(\lambda)}{10}$ and $\Lambda^*(x)=x\log (10 x)+(1-x)\log(10(1-x)/9)$. We use the 160 millions first decimals of Pi and the Euler constant available at \url{http://www.numberworld.org/constants.html} and \url{http://www.ginac.de/~kreckel/news.html}.

\medskip

At first we consider a realization $X_1,\cdots, X_N$ of $N=1.6\cdot10^8$ independent random variables uniformly distributed in $\{0,\dots, 9\}$, that are viewed as the $N$ first terms of the realization of an infinite  sequence of such independent variables $X_1,\cdots,X_n,\cdots$. In fact these digits are pseudo-random numbers provided by the Mersenne twister algorithm used in Matlab, so that actually we are also testing how such a sequence really looks like the theoretical one.

At each scale $n$, we choose a number of intervals $k(n) = \exp(n\Lambda^*( \Lambda' (\lambda_0)))$ with $\lambda_0 = 0.8$, so that $n\cdot k(n)\le N$ for $n\le 300$. Due to the fact that $\Phi_0(t)$ depends only on the first digit of $t$, $\Lambda_n^t$ is constant over the cylinder $[X_1\cdots X_{nk(n)}]$ which contains the random sequence $\widetilde D=X_1\cdots X_n\cdots $, and we can estimate it easily.   

Let $\lambda_{1}$ and $\lambda_2 $  the two solutions of the equation $\Lambda^*(\Lambda' (\lambda))=\Lambda^*(\Lambda' (0.8))$. One has $\lambda_1 \simeq -1.45$ and $\lambda_2 = \lambda_0 = 0.8$.

Figure~\ref{fig1}-$(a)$(left) illustrates the result of Theorem~\ref{th3}(1) and (3): the empirical logarithmic moment generating functions $\Lambda_n^{\widetilde D}$ converge to the function $\Lambda $ over the interval $(\lambda_1, \lambda_2)$, and on $(-\infty,\lambda_1]$ as well as  on $[\lambda_2,\infty)$, $\Lambda-\Lambda_n^{\widetilde D}$ converges to $\Lambda$ translated by an affine map.  Figure~\ref{fig1}-$(a)$(right) illustrates the same result in term of the Fenchel-Legendre transform $(\Lambda_n^{\widetilde D})^*$, which converges in the interval $(x_1, x_2)$, where $x_1 = \Lambda^{\prime}(\lambda_1) \simeq 0.0254$ and $x_2 = \Lambda^{\prime}(\lambda_2) \simeq  0.1983$ (the intervals of convergence are materialized by the dashed blue vertical lines). Moreover, on this figure one observes that the domain over which the functions $(\Lambda_n^{\widetilde D})^*$ are finite, which corresponds to $\overline{(\Lambda_n^{\widetilde D})'(\mathbb{R})}$, converges to the interval $[x_1,x_2]$. This is predicted by Theorem~\ref{th3}(2), since $\overline{(\Lambda_n^{\widetilde D})'(\mathbb{R})}$ is equal to the smallest closed interval containing $\{S_n\Phi(T^{n(j-1)}\widetilde D)/n:1\le j\le k(n)\}$.  

Figure~\ref{fig1}-$(b)$ numerically shows that, in terms of the convergence of the  logarithmic moment generating functions $\Lambda^D_{n}$ and their Fentchel-Legendre transform, the first $160$ million decimals of Pi  behave exactly like the previous sequence $X_1,\dots, X_N$ (though we do not expose the corresponding figures here, we verified that the same holds for all function $\Phi_a$, $a = 0, \cdots, 9$). The same conclusions hold for the $160$ millions first decimals of the Euler constant, as shown on Figure~\ref{fig1}-$(c)$.

\begin{figure}
\begin{tabular}{ccccc}
\multicolumn{5}{c}{$(a)$} \\
\multicolumn{5}{c}{i.i.d uniform sequence} \\
\rotatebox{90}{\hspace{1.7cm} $|\Lambda^{\widetilde D}_n -\Lambda|$} &
\includegraphics[width=0.4\textwidth]{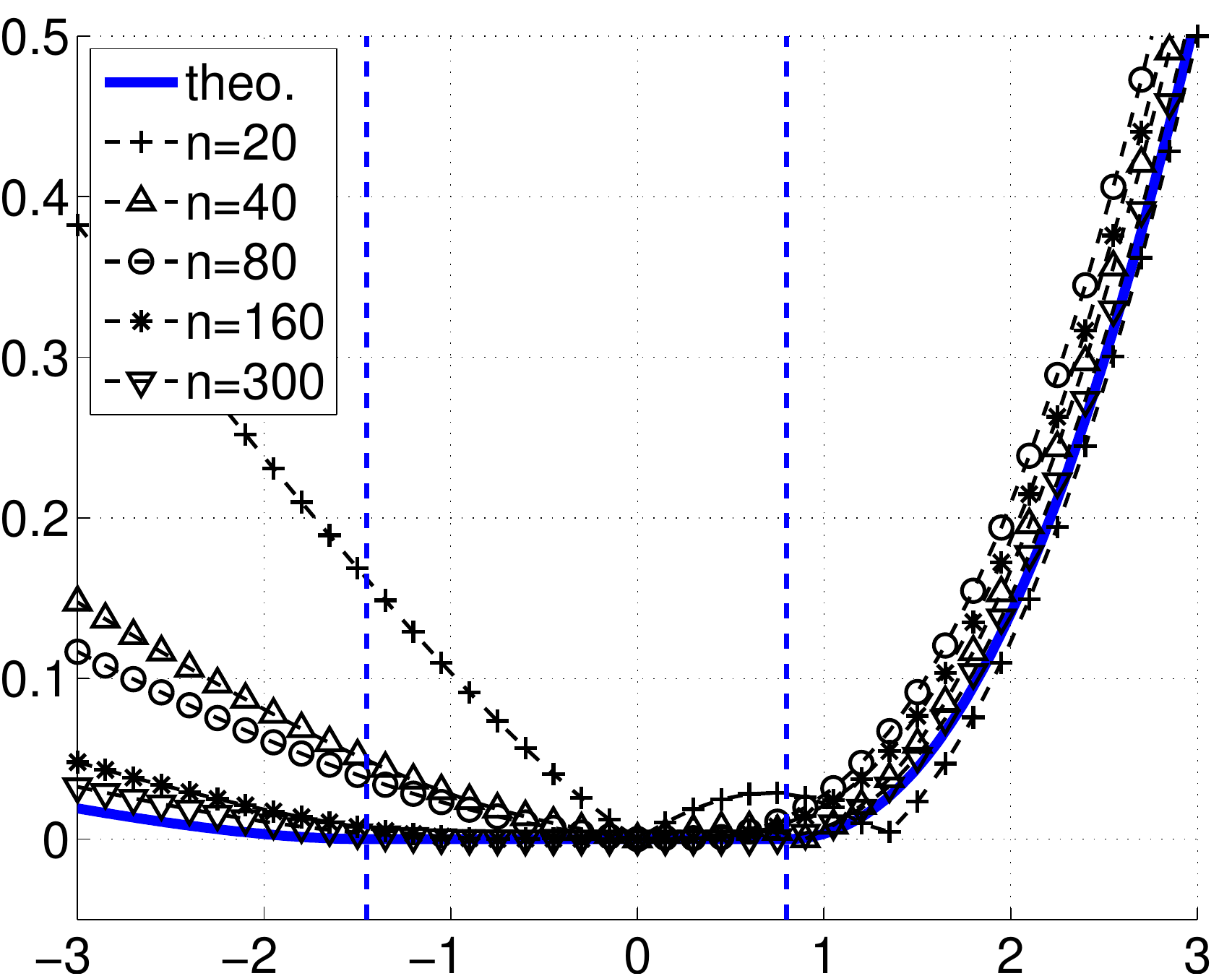} & &
\rotatebox{90}{\hspace{1.8cm} $(\Lambda^{\widetilde D}_n)^*, \Lambda^*$} &
\includegraphics[width=0.4\textwidth]{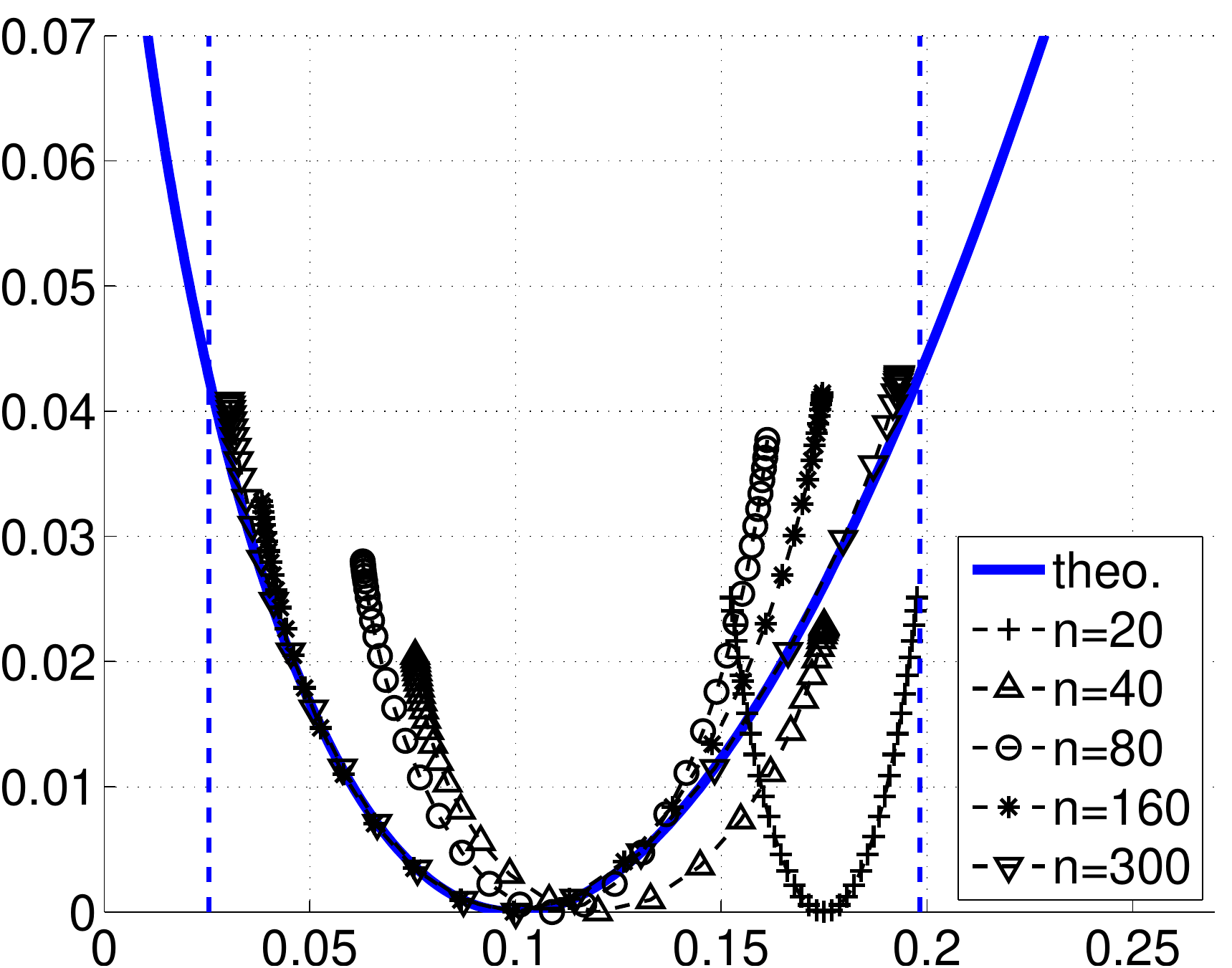} \\
& $\lambda$ & \hspace{0.2cm} & &  $x$ \\

\multicolumn{5}{c}{$(b)$} \\
\multicolumn{5}{c}{decimal digits of Pi} \\
\rotatebox{90}{\hspace{1.7cm} $|\Lambda^D_n -\Lambda|$} &
\includegraphics[width=0.4\textwidth]{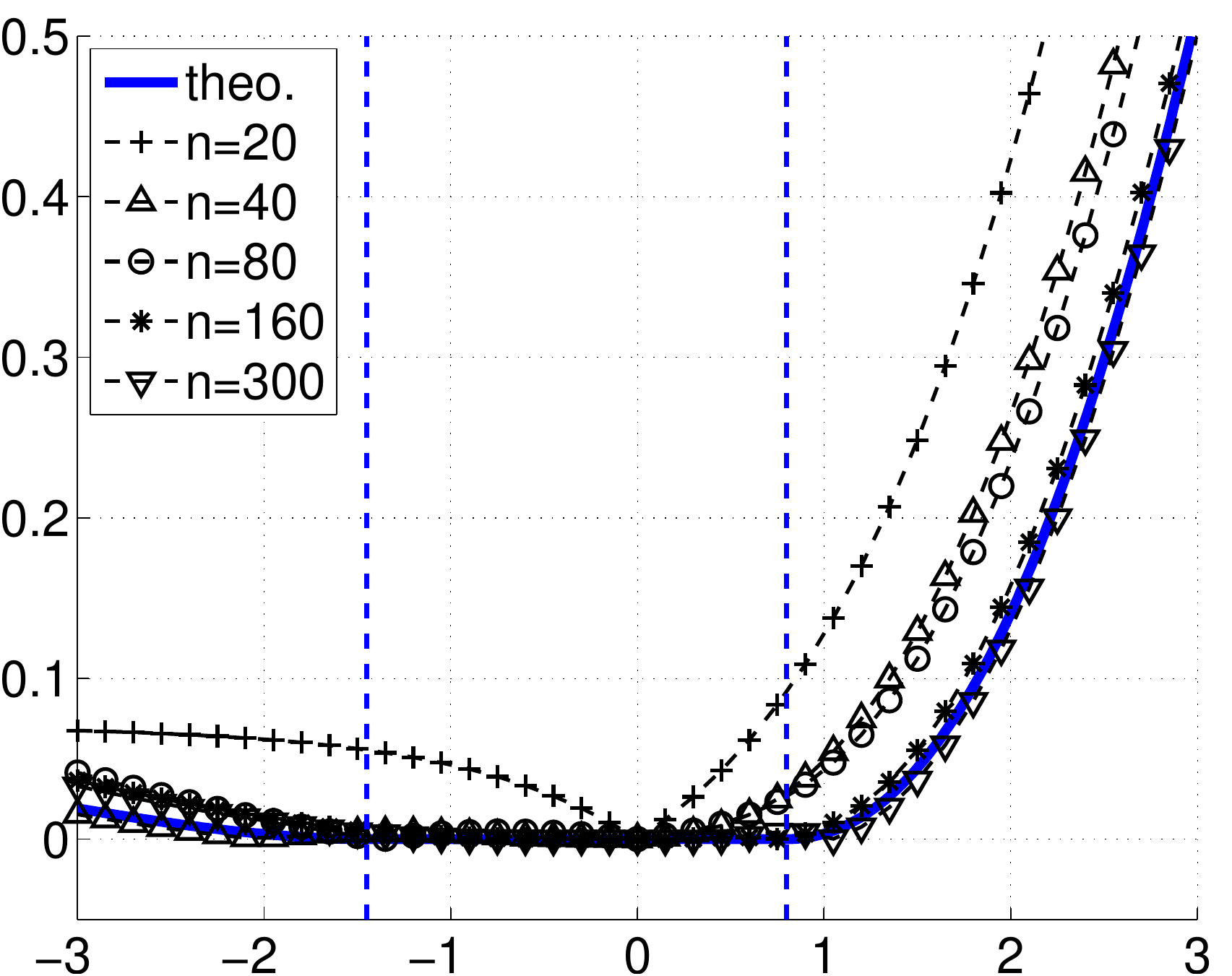} & &
\rotatebox{90}{\hspace{1.8cm} $(\Lambda^D_n)^*, \Lambda^*$} &
\includegraphics[width=0.4\textwidth]{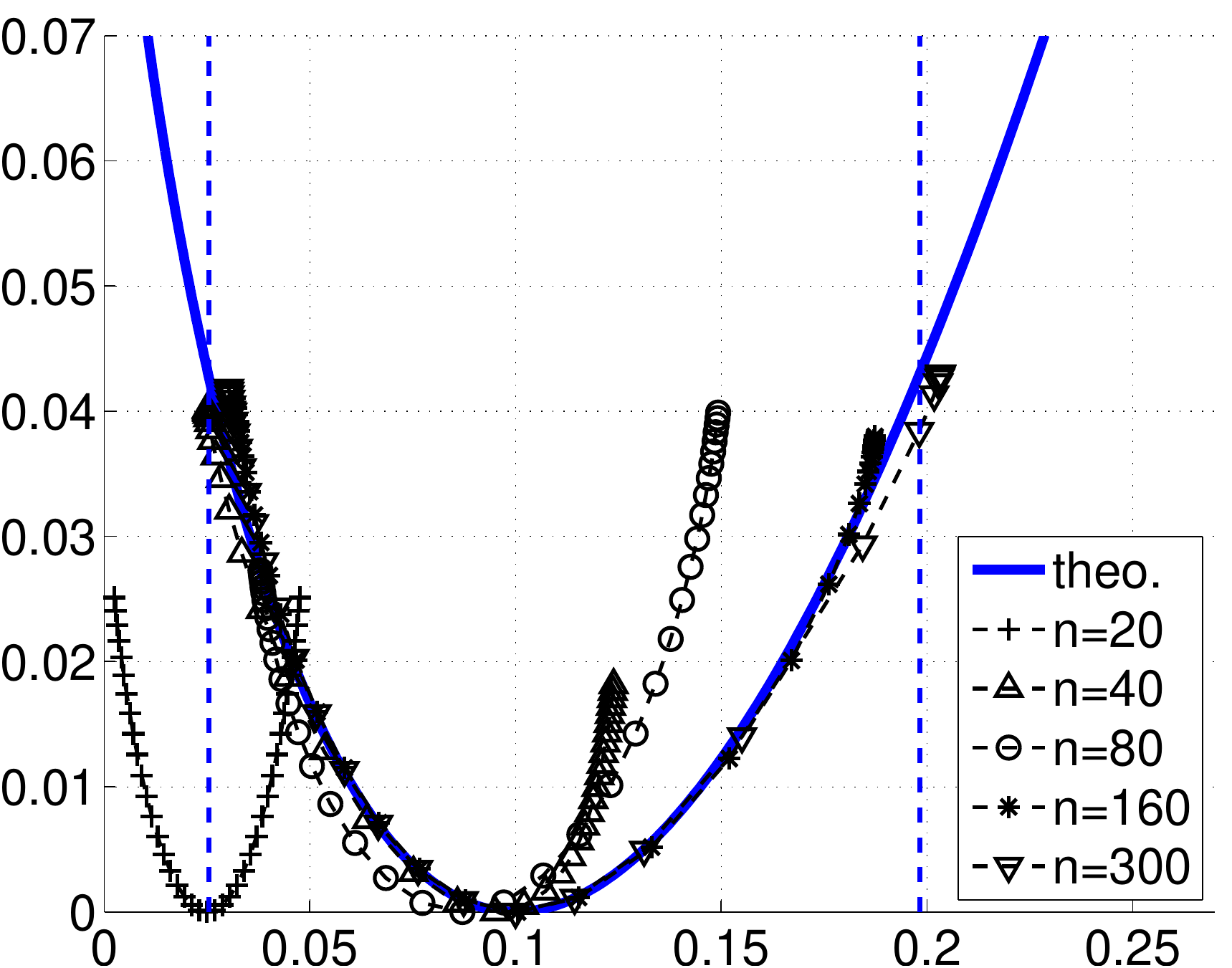} \\
& $\lambda$ & \hspace{0.2cm} & &  $x$ \\

\multicolumn{5}{c}{$(c)$} \\
\multicolumn{5}{c}{decimal digits of the Euler constant} \\
\rotatebox{90}{\hspace{1.7cm} $|\Lambda^D_n -\Lambda|$} &
\includegraphics[width=0.4\textwidth]{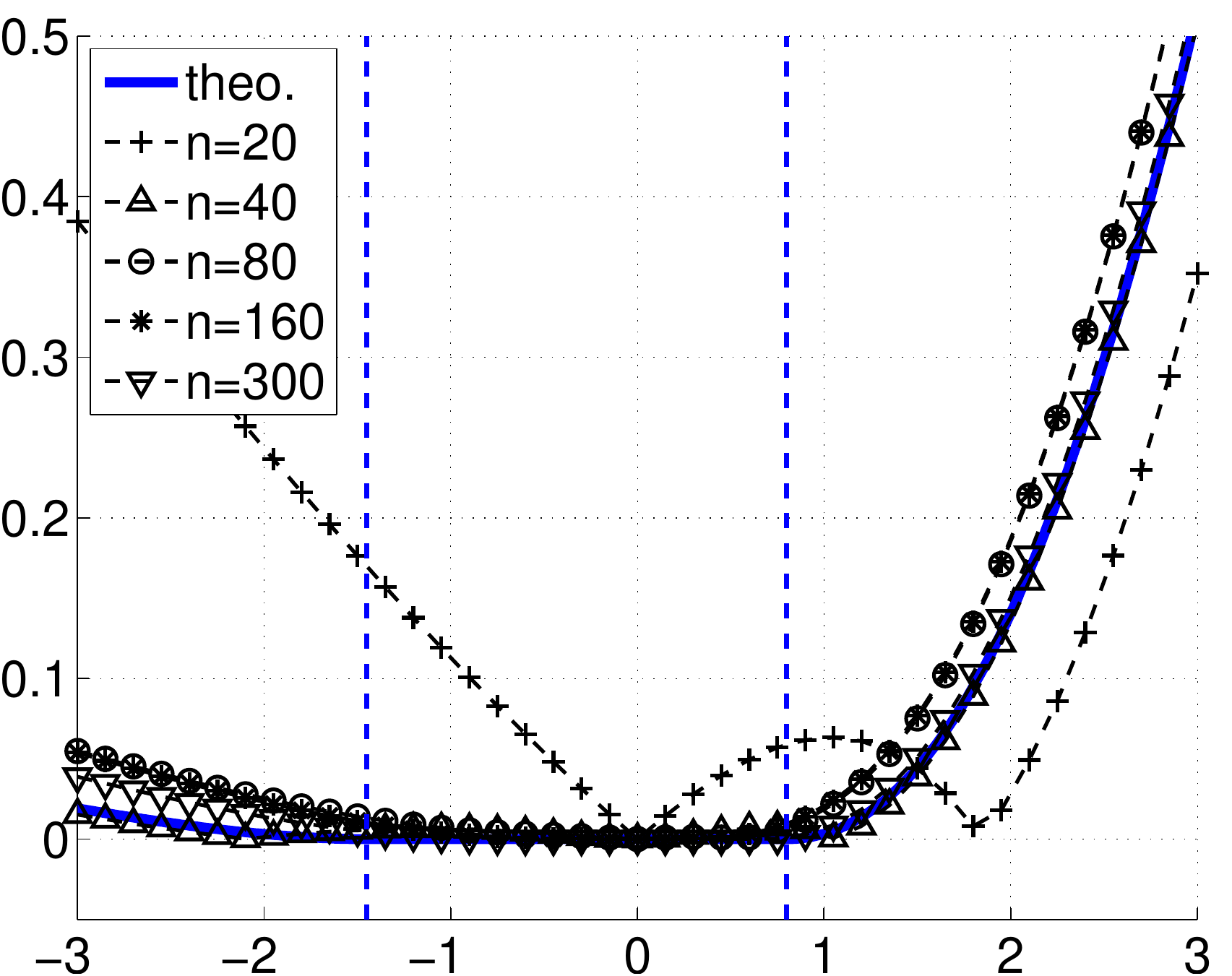} & &
\rotatebox{90}{\hspace{1.8cm} $(\Lambda^D_n)^*, \Lambda^*$} &
\includegraphics[width=0.4\textwidth]{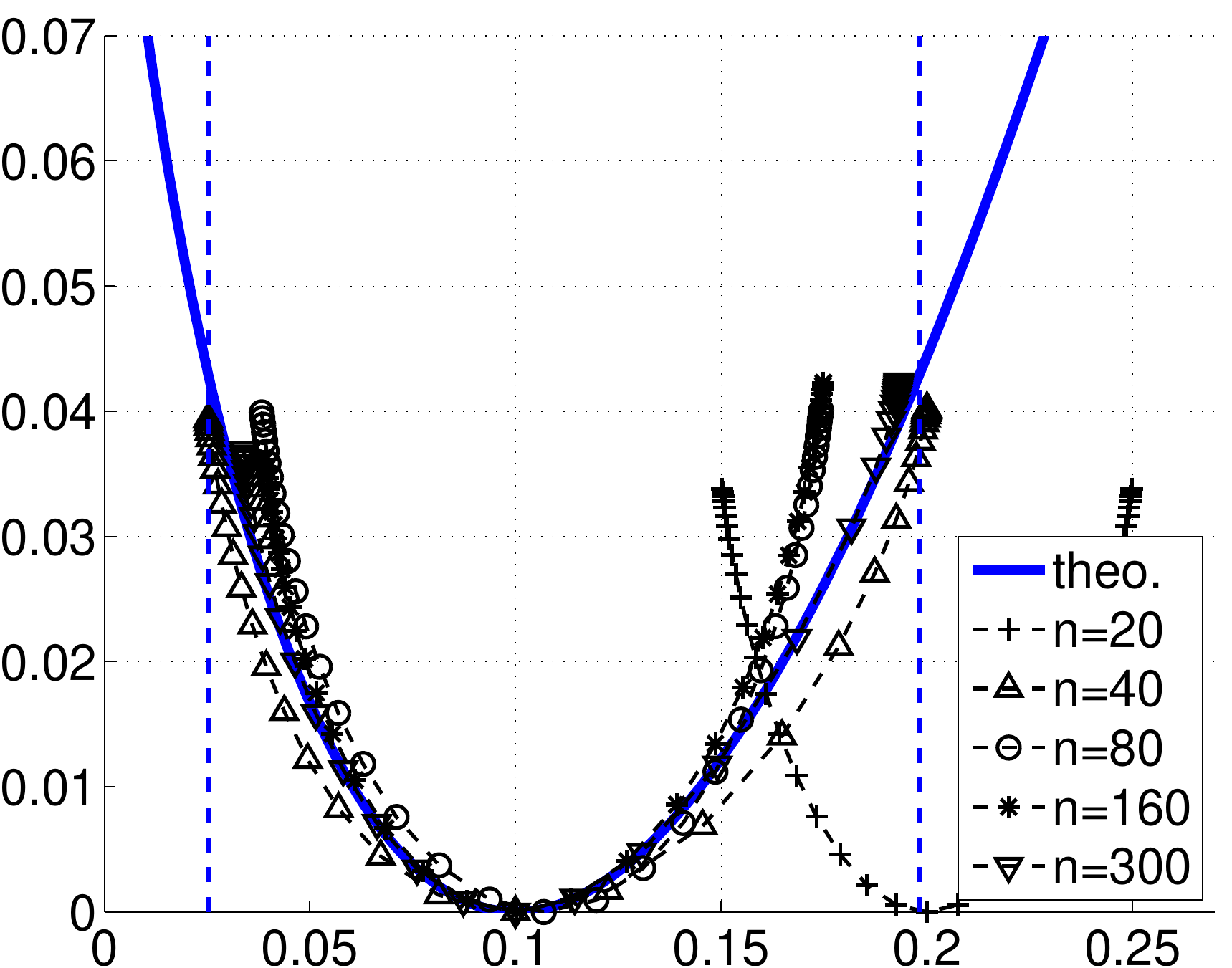} \\
& $\lambda$ & \hspace{0.2cm} & &  $x$ \\
\end{tabular}
\caption{Behavior of the logarithmic moment generating functions $\Lambda_n^{t}$ and their Fenchel-Legendre transform for $t$ equal to an i.i.d sequence of random variables uniformly distributed in $\{ 0, \cdots, 9\}$ and for $t$ equal to the the decimal digits of the number Pi and the Euler constant. }
\label{fig1}
\end{figure}

\section{Proofs of the main results}
\label{Proofs}
Recall that for any $n\ge 1$ and any compact subset $B$ of $\mathcal{D}$, $\delta_{n}\Lambda(B)=\sup\{|\Lambda(\lambda)-\Lambda_n(\lambda)|:\lambda\in B\}$, $\delta_n\Phi(B)=(\sup_{\lambda\in B}\|\lambda\|) \|S_n\Phi-S_n\Phi_n\|_\infty/n$, $\varLambda^*(B)=\sup\{\Lambda^*(\grad\Lambda(\lambda)):\lambda\in B\}$, $\xi_1(B)=\sup\{\|\grad\Lambda(\lambda)\|: \lambda\in B\}$, $\xi_2(B)=\sup\{\frac{1}{2}~^{t}\lambda {\rm D}^2\Lambda(\lambda)\lambda: \lambda\in B\}$, and $\xi(B)=\varLambda^*(B)+\xi_2(B)$. 
 
\medskip

The following lemma and corollary of its first part will be precious for us. The first part of the lemma can be found in (\cite{Rio00a}, p. 61), and the second one in \cite{BE}. Recall that given a real valued random variable $Y$ defined on $(\Omega,\mathcal A,\P)$, its quantile function $Q_Y$ is defined as the right-continuous inverse of the tail of $\P_{|Y|}$, the probability distribution of $|Y|$, {\em i.e.} 
$$
Q_Y(u)=\inf\{t\ge 0:\P(|Y|>t)\le u\} \quad (u\ge 0). 
$$
\begin{lem}\label{Rio0}
\begin{enumerate}
\item Let $(Y_j)_{j\ge 1}$ be a real valued and centered stationary process. 
For each $p\in (1,2)$ and $N\ge 1$ one has 
$$
\mathbb{E}\Big (\Big |\sum_{j=1}^NY_j\Big |^p\Big )\le C_p\,  N \int_0^1(\alpha^{-1}_{Y}(u))^{p-1} Q_Y(u)^p \, {\rm d}u,
$$ 
with $C_p=5^p\frac{p(5-2p)}{(p-1)(2-p)}$. 
\item Let $(Y_j)_{j\ge 1}$ be sequence of complex i.i.d. random variables. For each $p\in (1,2)$ and $N\ge 1$ one has 
$$
\mathbb{E}\Big (\Big |\sum_{j=1}^NY_j\Big |^p\Big )\le 2^p\,  N \mathbb{E}(|Y_1|^p).
$$ 
\end{enumerate}
\end{lem}
Then, the fact that $\int_0^1 Q_Y(u)^p \, du=\E |Y|^p$ together with H\"older's inequality yield
\begin{cor}\label{Rio}
Let $(Y_j)_{j\ge 1}$ be a real valued and centered stationary process. For each $p\in (1,2)$, $\epsilon>0$ and $N\ge 1$
$$
\mathbb{E}\Big (\Big |\sum_{j=1}^NY_j\Big |^p\Big )\le C_p N \left (\int_0^1(\alpha^{-1}_{Y}(u))^{(p-1)(1+\epsilon)/\epsilon} \, {\rm d}u\right)^{\epsilon/(1+\epsilon)} \big (\mathbb{E}|Y_1|^{(1+\epsilon)p} \big )^{1/(1+\epsilon)}.
$$
\end{cor}

We start with the most technical results, namely Theorem~\ref{th3},~\ref{thspeed} and~\ref{th2}.

\subsection{Proof of Theorem~\ref{th3}}

(1) For $\lambda\in\mathcal D$ and $n\ge 1$ we have 
\begin{equation}
\Lambda^\omega_n (\lambda)  =  \frac{1}{n} \log \frac{1}{k(n)}  \sum_{j=1}^{k(n)}  \exp \big (  \scal{\lambda}{S_n \Phi (T^{n(j-1)} X)}\big ) . 
\end{equation}
Fix $r_0>0$ such that $B(\lambda_0,r_0)\subset \mathcal D$. 
We must prove that we can find $r\in (0,r_0)$ such that  almost surely,  for all $\lambda\in B(\lambda,r)$, $\Lambda_n^\omega (\lambda)$ converges to $\Lambda (\mathbf {\lambda})$ as $n\to\infty$. 

\smallskip

 In order to exploit the mixing properties of the initial process $(X_1, \cdots)$, we use the uniform approximation of $S_n\Phi$ by the functions $S_n\Phi_n$.
For $\lambda\in\mathcal D$ and $n\ge 1$ let 
\begin{eqnarray*}
\Lambda^{(n),\omega}_n (\lambda)  =  \frac{1}{n} \log \frac{1}{k(n)}  \sum_{j=1}^{k(n)}  \exp \big (  \scal{\lambda}{S_n \Phi_n (T^{n(j-1)} X)}\big ) 
\end{eqnarray*}
and 
\begin{eqnarray*}
\nonumber \Lambda^{(n)}_n (\lambda) =  \frac{1}{n} \log \E \exp \big (  \scal{\lambda}{S_n \Phi_n (X)}\big ).
\end{eqnarray*}
By assumption {\bf (A3)} for all $\lambda\in B(\lambda_0,r_0)$ we have 
\begin{equation}\label{control1}
\max\big (|\Lambda^{(n),\omega}_n (\lambda) -\Lambda_n^\omega (\lambda)|, |\Lambda^{(n)}_n (\lambda) -\Lambda_n (\lambda)|\big )\le \delta_n\Phi(B(\lambda_0,r_0)), \text{ with } \lim_{n\to\infty}\delta_n\Phi(B(\lambda_0,r_0))=0.
\end{equation}
Consequently, it is enough to find $r\in (0,r_0)$ such that  almost surely, for all $\lambda\in B(\lambda_0,r)$ we have $\lim_{n\to\infty} \Lambda^{(n),\omega}_n (\lambda) - \Lambda^{(n)}_n (\lambda) =0$. 

\smallskip

For $\lambda\in B(\lambda_0,r_0)$ we write 
\begin{eqnarray}
 \nonumber \Lambda^{(n),\omega}_n (\lambda) &=& \Lambda^{(n)}_n (\lambda) +  \frac{1}{n} \log \left(       \frac{1}{k(n)}\sum_{j=1}^{k(n)} \exp\big ( \scal{\lambda}{ S_n \Phi_n(T^{n(j-1)} X)  } - n \Lambda^{(n)}_n (\lambda) \big )  \right), \\
\label{key} &=&\Lambda^{(n)}_n (\lambda) +  
  \frac{1}{n} \log \left ( 1 +    \frac{1}{k(n)}\sum_{j=1}^{k(n)} \Xn_j (\lambda) \right) ,
\end{eqnarray}
where 
\begin{equation}
\Xn_j (\lambda) =  \exp\big ( \scal{\lambda}{ S_n \Phi_n (T^{n(j-1)} X)  } - n \Lambda^{(n)}_n (\lambda) \big )  - 1.
\end{equation}
Now we notice that $(\Xn_j(\lambda))_{j\ge 1}$ is stationary and centered, and each $\Xn_j(\lambda)$ belongs to $\sigma (X^{(n)}_{j},X^{(n)}_{j+1})$ (recall \eqref{Xn}). Consequently, after writing 
$$
\mathbb{E}\Big (\Big |\sum_{j=1}^{k(n)} \Xn_j (\lambda)\Big |^p\Big )\le 2^{p-1} \mathbb{E}\Big (\Big |\sum_{1\le 2j\le k(n)}\Xn_{2j} (\lambda)\Big |^p\Big )+2^{p-1} \mathbb{E}\Big (\Big |\sum_{1\le 2j+1\le k(n)}\Xn_{2j+1} (\lambda)\Big |^p\Big ),
$$
we can apply Corollary~\ref{Rio} and get for $p\in (1,2)$ and $\epsilon >0$ 
\begin{equation}\label{esp}
\mathbb{E}\Big (\Big |\sum_{j=1}^{k(n)} \Xn_j (\lambda)\Big |^p\Big )\le M(n,p,\epsilon) k(n)\big (\mathbb{E}|\Xn_1(\lambda)|^{(1+\epsilon)p} \big )^{1/(1+\epsilon)}.
\end{equation}
where 
$$
M(n,p,\epsilon)=2^{p-1}C_p\left (\int_0^1(\alpha^{-1}_{X^{(n)}}(u))^{(p-1)(1+\epsilon)/\epsilon} \, {\rm d}u\right)^{\epsilon/(1+\epsilon)}.
$$ 

From now on we fix $p\in (1,2)$ close enough to one and $\epsilon_0>0$ small enough so that for all $\lambda\in B_0=B(\lambda_0,r_0/2)$ and $\epsilon\in (0,\epsilon_0)$ we have $p(1+\epsilon)\lambda\in \widetilde B_0=B(\lambda_0,r_0)$.

\medskip

We have (using successively the convexity of $u\ge 0\mapsto u^{(1+\epsilon)p}$ and the subadditivity of $v\ge 0\mapsto v^{1/(1+\epsilon)}$ to get the second and third lines)
\begin{eqnarray}
\nonumber && (\mathbb{E}|\Xn_1(\lambda)|^{(1+\epsilon)p} \big )^{1/(1+\epsilon)}\\&\le& \nonumber \Big (2^{p(1+\epsilon)-1}\big (\mathbb{E}\big [\exp\big ((1+\epsilon)p \scal{\lambda}{ S_n \Phi_n (T^{n(j-1)} X)  } - n(1+\epsilon)p  \Lambda^{(n)}_n (\lambda) \big ) \big ]+1\Big )^{1/(1+\epsilon)}\\
\nonumber &\le & 2^p\Big (1+\exp \Big [n\Big (\frac{\Lambda^{(n)}_n (p(1+\epsilon)\lambda)}{1+\epsilon}-p \Lambda^{(n)}_n (\lambda)\Big )\Big ]\Big )\\
\label{Z0n}&\le& 2^p\Big (1+\exp \Big [n\Big (\frac{\Lambda (p(1+\epsilon)\lambda)}{1+\epsilon}- p\Lambda(\lambda)+ 3\delta_n(\Lambda,\Phi)( \widetilde B_0)\Big )\Big ]\Big ),
\end{eqnarray}
where we have used {\bf (A1)} and {\bf (A3)}.  
Since $\Lambda$ is differentiable at $\lambda_0$, by using the first order Taylor expansion of $\Lambda$ at $\lambda_0$, for each $r\in (0,r_0/2)$, uniformly in $\lambda\in B(\lambda_0,r)$ we have 
\begin{eqnarray*}
\frac{\Lambda(p(1+\epsilon)\lambda)}{1+\epsilon}- p\Lambda(\lambda)
&=& (1-p)\big (\Lambda(\lambda_0)-\scal{\lambda_0}{\grad \Lambda (\lambda_0)}\big )+\xi(\epsilon)\epsilon+\eta(p,r)(p-1+r )\\
&=& (p-1)\Lambda^*(\grad \Lambda (\lambda_0))+\xi(\epsilon)\epsilon+\eta(p,r)(p-1+r ),
\end{eqnarray*}
where $\xi(\epsilon)$ is bounded over $(0,\epsilon_0)$ and $\eta(p,r)$ tends to $0$ as $p$ tends to $1^+$ and $r$ tends to $0^+$. This yields
%
\begin{eqnarray*}
&&(\mathbb{E}|\Xn_1(\lambda)|^{(1+\epsilon)p} \big )^{1/(1+\epsilon)}\\
&\le& 2^p\Big (1+\exp \big [n\big ((p-1)\Lambda^*(\grad \Lambda (\lambda_0))+\xi(\epsilon)\epsilon+\eta(p,r)(p-1+r )+ 3\delta_n(\Lambda,\Phi)( \widetilde B_0)\big )\big ]\Big )\\
&\le &2^{p+1} \exp \big [n\big ((p-1)\Lambda^*(\grad \Lambda (\lambda_0))+\xi(\epsilon)\epsilon+\eta(p,r)(p-1+r )+ 3\delta_n(\Lambda,\Phi)( \widetilde B_0)\big )\big ],
\end{eqnarray*}
the last inequality coming from the fact that ${\Lambda}^*\ge 0$. 

\medskip

Recall {\bf (A2)}. Let $h=(p-1)(1+\epsilon)/\epsilon$ and notice that $\alpha^{-1}_{X^{(n)}}\le \alpha^{-1}_{X}$. This yields
\begin{equation}\label{Mnp}
M(n,p,\epsilon)\le 2^{p-1}C_{p} M_{h}^{\epsilon/(1+\epsilon)}<\infty.
\end{equation}
Thus,  due to \eqref{esp}, for $n$ large enough so that $3\delta_n(\Lambda,\Phi)( \widetilde B_0)\le \epsilon$, uniformly in $\lambda\in B(\lambda_0,r)$ we have 
$$
\mathbb{E}\Big (\Big |\sum_{j=1}^{k(n)} \Xn_j (\lambda)\Big |^{p}\Big )\le 2^{2p} C_{p} M_{h}^{\epsilon/(1+\epsilon)}k(n) \exp \big [n\big ((p-1)\Lambda^*(\grad \Lambda (\lambda))+\widetilde \xi(\epsilon)\epsilon+\eta(p,r)(p-1+r )\big )\big ],
$$
where $\widetilde \xi(\epsilon)=\xi(\epsilon)+1$, hence 
\begin{eqnarray}
\label{estimp}&&\P \Big(  \Big| \frac{1}{k(n)}\sum_{j=1}^{k(n)} \Xn_{j}(\lambda) \Big| > \epsilon \Big) \le \epsilon^{-p}k(n)^{-p}\mathbb{E}\Big (\Big |\sum_{j=1}^{k(n)} \Xn_j (\lambda)\Big |^{p}\Big )\\
\nonumber &\le&2^{2p}  C_p M_{h}^{\epsilon/(1+\epsilon)}\epsilon^{-p}  k(n)^{1-p}\exp \big [n\big ((p-1)\Lambda^*(\grad \Lambda (\lambda))+\widetilde \xi(\epsilon)\epsilon+\eta(p,r)(p-1+r )\big )\big ]
.
\end{eqnarray}

Let $\eta>0$ such that $k(n)\ge \exp (n(\Lambda^*(\grad \Lambda (\lambda_0))+\eta))$ for $n$ large enough.  The previous inequality yields, for $n$ large enough, uniformly in $\lambda\in B(\lambda_0,r)$, 
\begin{eqnarray*}
\P \Big(  \Big| \frac{1}{k(n)}\sum_{j=1}^{k(n)} \Xn_{j}(\lambda) \Big| > \epsilon \Big)\le 2^{2p} C_p   M_{h}^{\epsilon/(1+\epsilon)}\epsilon^{-p} \exp \big (n\big ((1-p)\eta+\widetilde \xi(\epsilon)\epsilon+\eta(p,r)(p-1+r )\big )\big ).
\end{eqnarray*}
Hence, fixing $p$ close enough to 1 and  $r$ small enough so that $\eta(p,r)(p-1+r )\le (p-1)\eta/2$, we get 
$$
\P \Big(  \Big| \frac{1}{k(n)}\sum_{j=1}^{k(n)} \Xn_{j}(\lambda) \Big| > \epsilon \Big)\le 2^{2p} C_p   M_{h}^{\epsilon/(1+\epsilon)}\epsilon^{-p} \exp \big (n\big ((1-p)\eta/2+\widetilde \xi(\epsilon)\epsilon\big )\big ).
$$
Then,  for every $\epsilon$  small enough so that $\widetilde \xi(\epsilon)\epsilon\le (p-1)\eta/4$ we have
$$
\sum_{n\ge 1}\P \Big(  \Big| \frac{1}{k(n)}\sum_{j=1}^{k(n)} \Xn_{j}(\lambda) \Big| > \epsilon \Big)<\infty$$ 
for every $\lambda\in B(\lambda_0,r)$  (notice that at fixed $p$, $h$ tends to $\infty$ as $\epsilon$ tends to 0, this is why we need {\bf (A2)}). Now, by the Borel-Cantelli lemma and \eqref{key}, we can conclude that for every  $\lambda\in B(\lambda_0,r)$ we have $\lim_{n\to\infty} \Lambda^{(n),\omega}_n (\lambda) - \Lambda^{(n)}_n (\lambda) = 0$ almost surely, hence $\lim_{n\to\infty} \Lambda_n^\omega (\lambda)=\Lambda(\lambda)$ almost surely. From this we deduce that, with probability 1, $\lim_{n\to\infty} \Lambda_n^\omega (\lambda)=\Lambda(\lambda)$ for every $\lambda$ in a countable and dense subset of $B(\lambda_0,r)$. Since  the functions $\Lambda_n^\omega(\lambda)$ and $\Lambda$ are convex, we deduce from Theorem 10.8 in \cite{Roc} that almost surely, $\Lambda(\lambda)$ converges to $\Lambda(\lambda)$ for all $\lambda$ in $B(\lambda_0,r)$. 

\medskip

\noindent
(2) The first part is a direct consequence of (1) and Theorem~\ref{LDPloc}. 

Now let $x_0=\grad \Lambda (\lambda_0)$. Let $\eta>0$ such that $k(n)\le \exp(n(\Lambda^*(x_0)-\eta))$ for $n$ large enough. For $\epsilon$ small enough, if $n$ is large enough, we have 
\begin{eqnarray*}
\P\Big (\exists \, 1\le j\le k(n): \, \frac{S_n \Phi (T^{(j-1)n}X)}{n}\in B(x_0,\epsilon)\Big)&\le &k(n)\P\Big (\frac{S_n \Phi ( X)}{n}\in B(x_0,\epsilon)\Big)\\
&\le& \exp (-n\eta/2)
\end{eqnarray*}
by Theorem~\ref{LDPloc}. Thus, by the Borel-Cantelli lemma we see that if $\epsilon$ is small enough, with probability 1, for $n$ large enough $\big\{1\le j\le k(n): \, \frac{S_n \Phi (T^{(j-1)n}X)}{n}\in B(x_0,\epsilon)\big \}$ is empty.

\medskip

\noindent
(3) Let $P:t\ge 0\mapsto \Lambda (t\lambda_0)$. Without loss of generality we assume that $\lambda_0\neq 0$. Notice that $P$ is differentiable at 1 since $\Lambda$ is differentiable at $\lambda_0$, and our assumption on the strict convexity of $\Lambda$ implies that $\Lambda^*(x_0)=P^*(P'(1))>0$. Moreover, by our assumption on the strict convexity of $\Lambda$, at each point $t\in (0,1)$ at which $P$ is differentiable we have $P'(t)<P'(1)$ and $P^*(P'(t))<P^*(P'(1))$.  Consequently, if we set  $P_n^\omega(t)=\Lambda_n^\omega(t\lambda_0)$, we deduce from Theorem~\ref{th3}(1) that with probability 1, for all $t\in (0,1)$, $P_n^\omega(t)$ converges to $P(t)$ as $n$ tends to $\infty$.  Now, we notice that for any $s>1$, by the superadditivity of $y\ge 0\mapsto y^s$ and the definition of $P^\omega_n$, we have for $\eta\in (0,1)$ that $P_n^\omega(s(1-\eta))\le sP_n^\omega(1-\eta) +(s-1)\log (k_n)/n$. Thus, due to our assumption on $k(n)$, $\limsup_{n\to\infty}P_n^\omega(s(1-\eta))\le sP(1-\eta) +(s-1) \Lambda^*(x_0)$. If $t>1$, for each $\eta\in(0,1)$, if we set $s=t/(1-\eta)$, we get $\limsup_{n\to\infty}P_n^\omega(t)\le tP(1-\eta)/(1-\eta)+ (t-1+\eta)\Lambda^*(x_0)/(1-\eta)$.  
Consequently,  $\limsup_{n\to\infty}P_n^\omega(t)\le tP(1)+ (t-1)\Lambda^*(x_0)=\Lambda(\lambda_0)+(t-1)\scal{\lambda_0}{x_0}$.  

On the other hand, by convexity, for all $n\ge 1$ and $\eta\in (0,1)$, for $t>1$ we have $P_n^\omega(t)\ge P_n^\omega(1-\eta)+ (t-1+\eta) (P_n^\omega)'(1-\eta)$. Thus $\liminf_{n\to\infty}P_n^\omega(t)\ge P(1-\eta)+ (t-1+\eta) P'(1-\eta)$ for each $\eta$ so that $P'(1-\eta)$ exists. Letting $\eta$ go to 1, we get $\liminf_{n\to\infty}P_n^\omega(t)\ge P(1)+ (t-1) P'(1)=\Lambda(\lambda_0)+(t-1)\scal{\lambda_0}{x_0}$. Thus we have the conclusion.

\subsection{Proof of Theorem~\ref{thspeed}} Since $\Lambda$ is twice continuously differentiable,  by using the second order Taylor expansion of $\Lambda$ we can get for all $\lambda\in B_{\rho/2}$ and for all $p\in (1,2)$ and $\epsilon>0$ such that $p(1+\epsilon)B_{\rho/2}\subset B_\rho$ 
\begin{eqnarray*}
\frac{\Lambda(p(1+\epsilon)\lambda)}{1+\epsilon}- p\Lambda(\lambda)
= \frac{(p(1+\epsilon)-1)\Lambda^*(\grad \Lambda (\lambda))+\delta (p,\epsilon)}{1+\epsilon},
\end{eqnarray*}
where $|\delta(p,\epsilon)|\le \xi_2(B_\rho)(p(1+\epsilon)-1)^2$. Consequently, for all $\lambda\in B_{\rho/2}$ and for all $p\in (1,2)$ and $\epsilon\in (0,1/2)$ such that $p(1+\epsilon)B_{\rho/2}\subset B_\rho$,
\begin{multline*}
\Big |\frac{\Lambda(p(1+\epsilon)\lambda)}{1+\epsilon}- p\Lambda(\lambda)-(p-1)\Lambda^*(\grad \Lambda (\lambda))\Big | \\
\le \frac{ \Lambda^*(\grad \Lambda (\lambda))\epsilon+\delta (p,\epsilon)}{1+\epsilon}
\le  \Lambda^*(\grad \Lambda (\lambda))\epsilon+ \xi_2(B_\rho)(p-1)^2+ (p^2\epsilon+2p(p-1))\xi_2(B_\rho)\epsilon\\
\le \xi_2(B_\rho)(p-1)^2+(p^2\epsilon+2p(p-1)+1)(\xi_1(B_\rho)+\xi_2(B_\rho))\epsilon\\=\xi_2(B_\rho)(p-1)^2+(p^2\epsilon+2p(p-1)+1)\xi(B_\rho)\epsilon.
\end{multline*}
Thus, for $p$ close enough to $1$ and $\epsilon$ close enough to $0$, 
$$
\Big |\frac{\Lambda(p(1+\epsilon)\lambda)}{1+\epsilon}- p\Lambda(\lambda)-(p-1)\Lambda^*(\grad \Lambda (\lambda))\Big |\le \xi_2(B_\rho)(p-1)^2+2\xi(B_\rho)\epsilon.
$$
Let $(k(n))_{n\ge 1}$ and $(\epsilon_n)_{n\ge 1}$ be as in the statement, and take $\epsilon=\epsilon_n$ and $p=p_n=1+\sqrt{\epsilon_n}$. Defining the variables $\Xn_j$ as in the proof of Theorem~\ref{th3}, by using \eqref{Z0n} and \eqref{estimp} we can get 
\begin{eqnarray}\label{estimproba}
\P \Big(  \Big| \frac{1}{k(n)}\sum_{j=1}^{k(n)} \Xn_{j}(\lambda) \Big| > \epsilon_n \Big)\le 2^{2p_n} C_{p_n}   M_{h_n}^{\epsilon_n/(1+\epsilon_n)}\epsilon_n^{-p_n} \exp(-n\tau(n,\lambda))) \end{eqnarray}
with $h_n=(p_n-1)(1+\epsilon_n)/\epsilon_n$ and $\tau(n,\lambda)=\big ((p_n-1)(\log (k(n))/n-\Lambda^*(\grad \Lambda(\lambda))-\xi_2(B_\rho)(p_n-1)^2-2\xi(B_\rho)\epsilon_n-3\delta_n\Lambda(B_\rho)-3\delta_n\Phi(B_\rho)\big )$. We have (recall the value of $C_p$ given in Lemma~\ref{Rio0}(1))
$$
2^{2p_n} C_{p_n}  \epsilon_n^{-p_n}=O(\epsilon_n^{-3/2-\sqrt{\epsilon_n}})=O(\epsilon_n^{-3/2})
$$
as $n$ tends to $\infty$. Moreover, 
$$
\tau(n,\lambda)\ge \tau(n)= \sqrt{\epsilon_n}(\log (k(n))/n -\varLambda^*(B))-3(\xi(B_\rho)\epsilon_n+\delta_n\Lambda(B_\rho)+\delta_n\Phi(B_\rho)),
$$
and an estimation provided at the end of this proof shows that $M_{h_n}^{\epsilon_n/(1+\epsilon_n)}=O(1)$ as $n$ tends to $\infty$. Thus,  
$$
\P \Big(  \Big| \frac{1}{k(n)}\sum_{j=1}^{k(n)} \Xn_{j}(\lambda) \Big| > \epsilon_n \Big)=O\big(\epsilon_n^{-3/2}\exp (-n\tau(n))\big ).
$$ 
Now, let $g(n)=\lfloor \log_2(\sqrt{d}/\epsilon_n)\rfloor +1$ and $\mathcal{G}_n(B_{\rho/2})=\{(k_1,\dots,k_d)\in \Z^d: (k_12^{-g(n)},\dots,k_d2^{-g(n)})\in B_{\rho/2}\}$. There exists a constant $C(B_{\rho/2})$ depending on the volume of $B_{\rho/2}$ only such that $\#\mathcal{G}_n(B_{\rho/2})\le C(B_{\rho/2})\epsilon_n^{-d}$, hence
$$
\P \Big( \exists \lambda\in\mathcal{G}_n(B_{\rho/2}): \Big| \frac{1}{k(n)}\sum_{j=1}^{k(n)} \Xn_{j}(\lambda) \Big| > \epsilon_n \Big)=O\big(\epsilon_n^{-(3/2+d)}\exp (-n\tau(n))\big ),
$$
and due to \eqref{BC}, the Borel-Cantelli lemma ensures that, with probability 1, for $n$ large enough, for all $\lambda\in\mathcal{G}_n(B_{\rho/2})$, $\Big| \frac{1}{k(n)}\sum_{j=1}^{k(n)} \Xn_{j}(\lambda) \Big| \le \epsilon_n$. This can  be used in \eqref{key} and combined with \eqref{control1} to get for $n$ large enough
\begin{equation}\label{dyadique}
\sup_{\lambda_n\in \mathcal{G}_n(B_{\rho/2})}|\Lambda_n^\omega(\lambda_n)-\Lambda(\lambda_n)|\le (\epsilon_n+\epsilon_n^2)/n+ \delta_n\Lambda(B_\rho)+\delta_n\Phi(B_\rho).
\end{equation}
Since $(\mathcal{G}_n(B_{\rho/2}))_{n\ge 1}$ is increasing and $\bigcup_{n\ge 1}\mathcal{G}_n(B_{\rho/2})$ is dense in $B_{\rho/2}$, the convexity of $\Lambda_n^\omega$ and $\Lambda$ ensures that $\Lambda_n^\omega$ converges uniformly to $\Lambda$ over $B_{\rho/2}$, and $\grad \Lambda_n^\omega$ converges uniformly to $\grad \Lambda$ over $B$ (see \cite{Roc}, Th. 10.8 and 25.7). Thus, for any $\eta>0$, if $n$ is large enough, we have both \eqref{dyadique} and $\sup_{\lambda\in B} \|\grad \Lambda_n^\omega-\grad \Lambda\|\le \eta/2$, so that for all $\lambda\in B$, we can choose $\lambda_n\in \mathcal{G}_n(B)$ such that $\|\lambda-\lambda_n\|\le \epsilon_n$, hence 
\begin{eqnarray*}
|\Lambda_n^\omega(\lambda)-\Lambda(\lambda)|&\le& |\Lambda_n^\omega(\lambda_n)-\Lambda(\lambda_n)|+|\Lambda_n^\omega(\lambda)-\Lambda_n^\omega(\lambda_n)|+
|\Lambda(\lambda)-\Lambda(\lambda_n)|\\
&\le& (\epsilon_n+\epsilon_n^2)/n+\delta_n\Lambda(B_\rho)+\delta_n\Phi(B_\rho)+(\eta/2+2\max_{\lambda'\in B}\|\grad \Lambda\|) \|\lambda-\lambda_n\|\\
&\le& (\eta+2\max_{\lambda'\in B}\|\grad \Lambda\|)\epsilon_n+\delta_n\Lambda(B_\rho)+\delta_n\Phi(B_\rho). 
\end{eqnarray*}

It remains to prove that $M_{h_n}^{\epsilon_n/(1+\epsilon_n)}=O(1)$ as $n$ tends to $\infty$. Due to {\bf (A2')}, there exists $C>0$ such that $\alpha_X(u)^{-1}\le C|\log(u)|^{1/\theta}$ for all $u\in (0,1]$. This yields for $h>0$
\begin{multline*}
M_h=\int_0^1(\alpha_X(u)^{-1})^h\,{\rm d}u\le \int_0^1C^h|\log(u)|^{h/\theta}\,{\rm d}u\\=C^h\Gamma(1+h/\theta)=O\Big(C^h (N(h,\theta)/ e)^{N(h,\theta)}\sqrt{2\pi N(h,\theta)}\Big ),
\end{multline*}
where $N(h,\theta)=\lfloor h/\theta\rfloor +1$ and we have use Stirling's formula. 

Now, we can use the fact that $h_n=(1+\epsilon_n)/\sqrt{\epsilon_n}$ and the estimate above to conclude that $M_{h_n}^{\epsilon_n/(1+\epsilon_n)}=O(1)$ as $n$ tends to $\infty$.

\subsection{Proof of Theorem~\ref{th2}} Here we have $\delta_n\Lambda=\delta_n\Phi=0$, so that with respect to the proof of Theorem~\ref{thspeed}, we can consider  the centered, independent and identically distributed variables 
$$
Z_{n,j} (\lambda) =  \exp\big ( \scal{\lambda}{ S_n \Phi (T^{nj} X)  } - n \Lambda (\lambda) \big )  - 1,
$$ 
instead of the $Z^{(n)}_j (\lambda)$, with $\Lambda(\lambda)=\log \E(\exp\scal{\lambda}{X})$. Now we can use Lemma~\ref{Rio}(2) instead of Lemma~\ref{Rio}(1). This yields, for $p$ small enough so that $pB\subset B_\rho$ and $\lambda\in B$
\begin{eqnarray*}
\mathbb{E}\Big (\Big |\sum_{j=1}^{k(n)} Z_{n,j} (\lambda)\Big |^p\Big )&\le &2^p k(n)\big (\mathbb{E}|Z_{n,1}(\lambda)|^{p} \big )\\
&\le& 2^{2p-1} k(n)\big (1+ \E \exp(p\scal{\lambda }{ S_n \Phi (T^{nj} X)} -np\Lambda(\lambda) )\big)\\
&=&2^{2p-1} k(n)\big (1+\exp\big ( n(\Lambda(p\lambda)-p\Lambda(\lambda))\big )\big )\\
&\le &2^{2p}k(n) \exp \big (n[(p-1)\Lambda^*(\grad\Lambda (\lambda))+ (p-1)^2\xi_2(B_\rho)]\big)  .
\end{eqnarray*}
The proof finishes as that of Theorem~\ref{thspeed}(1).

\subsection{Proof of Theorem~\ref{ASLDP}} (1) Fix $\theta>0$. For $n\ge 1$, we denote by $\mu^{(n)}_n$ the probability distribution of $S_n\Phi_n(X)/n$ and by $\mu_n^{(n),\omega}$ the empirical distribution of $(S_n\Phi_n(T^{(j-1)}X(\omega))/n$. 

Recall that $(\delta_n)_{n\ge 1}$ is defined in {\bf (A3)}. Let $B_n\in\{B(x,r-\delta_n),B(x,r+\delta_n)\}$. We can estimate $\P\big(\big |\mu_n^{(n),\omega}(B_n)-\mu_n^{(n)}(B_n)\big |\ge \theta\mu_n^{(n)}(B_n)\big )$ as when we get \eqref{esp} thanks to Corollary~\ref{Rio}. 

Fix $p\in (1,2)$, $\epsilon>0$ and $h=(p-1)(1+\epsilon)/\epsilon$.  We have 
\begin{eqnarray*}
&&\P\big(\big |\mu_n^{(n),\omega}(B_n)-\mu_n^{(n)}(B_n)\big |\ge\theta\mu_n^{(n)}(B_n)\big )\\
&\le & (\theta\mu_n^{(n)}(B_n))^{-p}\mathbb{E}\Big( \frac{1}{k(n)^p}\Big|\sum_{j=1}^{k(n)}\mathbf{1}_{B_n}(S_n\Phi_n(T^{(j-1)n}X)/n)-\P(S_n\Phi_n(X)/n\in B_n )\Big|^p\Big)\\
&\le &2^{p-1}C_pM_h^{\epsilon/(1+\epsilon)} (\theta\mu_n^{(n)}(B_n))^{-p}
k(n)^{1-p} \mathbb{E}\Big(\Big|\mathbf{1}_{B_n}(S_n\Phi_n(X)/n)-\P(S_n\Phi_n(X)/n\in B_n )\Big|^{p(1+\epsilon)}\Big)^{1/(1+\epsilon)}\\
&\le & 2^{p-1}C_pM_h^{\epsilon/(1+\epsilon)} (\theta\mu_n^{(n)}(B_n))^{-p}
k(n)^{1-p}\\
& & \quad \cdot  \Big(2^{p(1+\epsilon)-1}(\mathbb{P}(S_n\Phi_n(X)/n\in B_n)+\P(S_n\Phi_n(X)/n\in B_n )^{p(1+\epsilon)}\Big)^{1/(1+\epsilon)}\\
&\le& 2^{2p-1}C_pM_h^{\epsilon/(1+\epsilon)} \theta^{-p}k(n)^{1-p} \mu^{(n)}_n(B_n)^{-p} (\mu^{(n)}_n(B_n)^{1/(1+\epsilon)}+\mu^{(n)}_n(B_n)^p)\\
&\le & 2^{2p}C_p M_h^{\epsilon/(1+\epsilon)} \theta^{-p}k(n)^{1-p} \mu^{(n)}_n(B_n)^{-p+1/(1+\epsilon)}\\
&\le& 2^{2p}C_p M_h^{\epsilon/(1+\epsilon)} \theta^{-p}k(n)^{1-p}\mu_n(B(x,r-2\delta_n))^{-p+1/(1+\epsilon)}.
\end{eqnarray*}
Now suppose that $\displaystyle \liminf_{n\to\infty} \frac{\log k(n)}{n}>I(x)$ and let $\eta>0$ such that $k(n)\ge \exp(n(I(x)+\eta))$ for $n$ large enough. Since $(\mu_n)_{n\ge 1}$ satisfies the LDP with rate function $I$, for $n$ large enough we have $\mu_n(B(x,r-2\delta_n))\ge \exp (-n(I(x)+\eta/4))$. Consequently, we can choose $\epsilon$ small enough so that $k(n)^{1-p}\mu_n(B(x,r-2\delta_n))^{-p+1/(1+\epsilon)}\le \exp (-n(p-1)\eta/2)$ for $n$ large enough. Then, the previous bound for $\P\big(\big |\mu_n^{(n),\omega}(B_n)-\mu_n^{(n)}(B_n)\big |\ge\theta\mu_n^{(n)}(B_n)\big )$ yields $\sum_{n\ge 1}  \P\big(\big |\mu_n^{(n),\omega}(B_n)-\mu_n^{(n)}(B_n)\big |\ge\theta\mu_n^{(n)}(B_n)\big )<\infty$ for any $\theta>0$. Hence, by the Borel-Cantelli lemma we get that with probability one,  $\lim_{n\to\infty}\frac{1}{n}\log \frac{\mu_n^{(n),\omega}(B_n)}{\mu_n^{(n)}(B_n)}=0$ for $B_n\in\{B(x,r-\delta_n),B(x,r+\delta_n)\}$. 

Moreover,  we have $\mu_n^{(n),\omega}(B(x,r-\delta_n))\le \mu^\omega_n(B(x,r))\le \mu_n^{(n),\omega}(B(x,r+\delta_n))$, and on the other hand we have $\mu_n(B(x,r-2\delta_n))\le  \mu_n^{(n)}(B(x,r-\delta_n))\le \mu_n(B(x,r))\le \mu_n^{(n)}(B(x,r+\delta_n))\le \mu_n(B(x,r+2\delta_n))$. 

Consequently, for any $r>0$, with probability 1, 
\begin{multline*}
\liminf_{n\to\infty}-\frac{1}{n}\log \mu_n(B(x,r+2\delta_n))\le \liminf_{n\to\infty}-\frac{1}{n}\log \mu^\omega_n(B(x,r))\\\le \limsup_{n\to\infty}-\frac{1}{n}\log \mu^\omega_n(B(x,r))\le\limsup_{n\to\infty}-\frac{1}{n}\log \mu_n(B(x,r-2\delta_n)).
\end{multline*}
This implies that with probability 1, for all $r\in\mathbb {Q}_+^*$  we have 
\begin{multline*}
\liminf_{n\to\infty}-\frac{1}{n}\log \mu_n(B(x,3r/2))\le \liminf_{n\to\infty}-\frac{1}{n}\log \mu^\omega_n(B(x,r))\\\le \limsup_{n\to\infty}-\frac{1}{n}\log \mu^\omega_n(B(x,r))\le\limsup_{n\to\infty}-\frac{1}{n}\log \mu_n(B(x,r/2)).
\end{multline*}
But since  $(\mu_n)_{n\ge 1}$ satisfies the LDP with rate function $I$, we have for all $y\in\mathcal Y$ (see \cite{DZ}, Th. 4.1.18)
\begin{equation}\label{LDPgene}
\displaystyle \lim_{s\to 0^+}-\liminf_{n\to\infty} \frac{1}{n}\log\mu_n(B(y,s))=\lim_{s\to 0^+}-\limsup_{n\to\infty} \frac{1}{n}\log\mu_n(B(y,s))=I(y).
\end{equation}
This, together with the previous inequalities yields the desired result. 

At last, suppose that $\displaystyle \limsup_{n\to\infty} \frac{\log k(n)}{n}<I(x)$. An estimate similar to that used to establish the second part of Theorem~\ref{th3}(2) yields the desired result.

\medskip

\noindent
(2) Let $B \in B_{\mathcal Y}$ and $\gamma> 0$. If $\mu_n (B)>0$ we have 
\begin{eqnarray*}
\P(\mu_n^{\omega}(B)> \exp( n\gamma)\mu_n(B)  )&\le& 
\P\Big (\sum_{j=1}^{k(n)} \mathbf{1}_{B}(S_n\Phi(T^{(j-1)n}X)/n)> k(n) \exp( n\gamma) \mu_n(B)\Big ) \\
&\le & k(n)^{-1} \exp( -n\gamma) \mu_n(B)^{-1}\mathbb{E} \Big (\sum_{j=1}^{k(n)} \mathbf{1}_{B}(S_n\Phi(T^{(j-1)n}X)/n)\Big )\\
&=&\exp( -n\gamma),
\end{eqnarray*}
and clearly if $\mu_n(B)=0$ then $\mu^\omega_n(B)=0$ almost surely so that we also have  $\P(\mu_n^{\omega}(B)> \exp( n\gamma)\mu_n(B) )\le\exp( -n\gamma) $. 

Now fix $r>0$ and take $B=B(x,r)$. Since $\sum_{n\ge 1} \exp( -n\gamma) <\infty$, the Borel-Cantelli lemma yields $\limsup_{n\to\infty}\frac{1}{n}\log \mu^{\omega}_n(B(x,r))\le \gamma+\limsup_{n\to\infty}\frac{1}{n}\log \mu_n(B(x,r))$ almost surely. This holds for all $\gamma>0$, so $\limsup_{n\to\infty}\frac{1}{n}\log \mu^{\omega}_n(B(x,r))\le\limsup_{n\to\infty}\frac{1}{n}\log \mu_n(B(x,r))$ almost surely. This is enough to conclude thanks to \eqref{LDPgene} and the fact that $I(x)=\infty$. 

\medskip

\noindent (3) Let $\alpha<\infty$ and $K_{\alpha}\subset \mathcal Y$, a compact set such that $\limsup_{n\to\infty} \frac{1}{n}\log\mu_n(K_\alpha^c)\le -2\alpha$. By using the estimate obtained above with $B=K_\alpha^c$  and $\gamma=\alpha$ we get that with probability 1,  $\limsup_{n\to\infty}  \frac{1} {n}\log \mu^\omega_n(K_\alpha^c)\le -\alpha$.

\subsection{Proof of Theorem~\ref{corASLDP}} Our goal is to prove that, with probability 1, for all $y\in \mathcal Y$ we have 
\begin{equation}\label{LDPomega}
\lim_{r\to 0^+} \liminf_{n\to\infty}-\frac{1}{n}\log \mu^\omega_n(B(y,r))=\lim_{r\to 0^+} \limsup_{n\to\infty}-\frac{1}{n}\log \mu^\omega_n(B(y,r))=I(y).
\end{equation}
Then, due to Theorem 4.1.11 in \cite{DZ}, we have the desired almost sure weak LDP.
  
\medskip

Let $\mathcal{D}$ be a dense countable subset of $\mathcal{D}_I$. We can deduce from the end of the proof of Theorem~\ref{ASLDP}(1) that there exists a measurable subset $\Omega'$ of $\Omega$ such that $\P(\Omega')=1$ and for all $\omega\in\Omega'$, for all $x\in\mathcal{D}$ and for all $r\in\mathbb {Q}_+^*$ we have 
\begin{multline*}
\liminf_{n\to\infty}-\frac{1}{n}\log \mu_n(B(x,3r/2))\le \liminf_{n\to\infty}-\frac{1}{n}\log \mu^\omega_n(B(x,r))\\\le \limsup_{n\to\infty}-\frac{1}{n}\log \mu^\omega_n(B(x,r))\le\limsup_{n\to\infty}-\frac{1}{n}\log \mu_n(B(x,r/2)).
\end{multline*}
Now let $y\in\mathcal{D}_I$. For all $s>0$ we can find $x\in\mathcal{D}$ as well as a rational number $0<r<s$ such that $  B(y, s/4)\subset B(x, r/2) \subset B(y,s)\subset B(x,3/2r) \subset  B(y,2s)$. Consequently, for all $\omega\in \Omega'$, $y\in\mathcal{D}_I$ and $r>0$ we have 
\begin{multline*}
\liminf_{n\to\infty}-\frac{1}{n}\log \mu_n(B(y,2s))\le \liminf_{n\to\infty}-\frac{1}{n}\log \mu^\omega_n(B(y,s))\\\le \limsup_{n\to\infty}-\frac{1}{n}\log \mu^\omega_n(B(y,s))\le\limsup_{n\to\infty}-\frac{1}{n}\log \mu_n(B(y,s/4)).
\end{multline*}
Due to \eqref{LDPgene},  for all $\omega\in \Omega'$ and $y\in\mathcal{D}_I$ we get 
$$
\lim_{s\to 0^+} \liminf_{n\to\infty}-\frac{1}{n}\log \mu^\omega_n(B(y,s))=\lim_{s\to 0^+} \limsup_{n\to\infty}-\frac{1}{n}\log \mu^\omega_n(B(y,s))=I(y),
$$
that is \eqref{LDPomega} for $y\in \mathcal{D}_I$. 

Now suppose that $\mathcal Y\setminus \mathcal{D}_I\neq \emptyset$ and let $\mathcal{D}'$ be a dense subset of $\mathcal Y\setminus \mathcal{D}_I$.  Due to the facts established in the proof of Theorem~\ref{ASLDP}(2), there exists a measurable subset $\Omega'$ of $\Omega$ such that $\P(\Omega')=1$ and for all $\omega\in\Omega'$, for all $x\in\mathcal{D}'$, for all  $r\in\mathbb {Q}_+^*$ we have $\limsup_{n\to\infty}\frac{1}{n}\log \mu^{\omega}_n(B(x,r))\le \limsup_{n\to\infty}\frac{1}{n}\log \mu_n(B(x,r))$. 

Now, for all $y\in \mathcal Y\setminus \mathcal{D}_I$ and $s>0$, we can find $x\in \mathcal{D}'$ and $0<s<r\in \mathbb Q$ such that $B(y,s)\subset B(x,r)\subset B(y,2s)$, and the previous inequality yields, for all $\omega'\in\Omega'$, $\limsup_{n\to\infty}\frac{1}{n}\log \mu^{\omega}_n(B(y,s))\le \limsup_{n\to\infty}\frac{1}{n}\log \mu_n(B(y,2s))$. This yields \eqref{LDPomega}.

\section{Proofs of Theorem~\ref{brown} and \ref{brownfonc}}\label{pfbrown}

\subsection{Proof of Theorem~\ref{brown}}\label{pfbrown1}
Each interval $J_{k(n),j}$ can be decomposed into a union of $n$ consecutive closed intervals $J_{k(n),j,i}$ of length $1/nk(n)$. The increments $\Delta W(J_{k(n),j,i})$ take the form $(nk(n))^{-1/2} X_{k(n),j,i}$, where $ ( X_{k(n),j,i})_{\substack{1\le j\le k(n)\\1\le i\le \kappa(n)}}$ is a family of $nk(n)$ centered Gaussian vectors of covariance matrix the identity. Thus $\Delta W (J_{k(n),j})=(nk(n))^{-1/2} S_n(j)$ with $S_n(j)=\sum_{i=1}^{n} X_{k(n),j,i}$. Let 
$$
Z_{n,j}(\lambda)= \exp\big ( \scal{\lambda}{ S_n(j)} - n \Lambda (\lambda) \big )  - 1,
$$
with $\Lambda(\lambda)=\log \E(\exp\scal{\lambda}{X_{k(n),j,i}})=\|\lambda\|^2/2$, hence $\Lambda^*(\grad\Lambda(\lambda))=\|\lambda\|^2/2$ and $~^t\lambda {\rm D}^2\Lambda(\lambda)\lambda=\|\lambda\|^2$ for all $\lambda\in \R^d$. As in the proof of Theorem~\ref{th2} we have 
\begin{eqnarray*}
\mathbb{E}\Big (\Big |\sum_{j=1}^{k(n)} Z_{n,j} (\lambda)\Big |^p\Big )\le 2^{2p-1} k(n)\big (1+\exp\big ( n(\Lambda(p\lambda)-p\Lambda(\lambda))\big )\big )
\end{eqnarray*}
which, due to the special form of $\Lambda$, yields 
\begin{eqnarray*}
\mathbb{E}\Big (\Big |\sum_{j=1}^{k(n)} Z_{n,j} (\lambda)\Big |^p\Big )\le 2^{2p}k(n) \exp \big (n[(p-1)\|\lambda\|^2/2+ (p-1)^2\|\lambda\|^2/2]\big).   
\end{eqnarray*}
Then, we can use the same approach as that used in the proof of Theorem~\ref{thspeed} to get that under \eqref{condbrown}, with probability 1, 
$$
\lim_{n\to\infty} \Big (\Lambda^\omega_n(\lambda)=\frac{1}{n} \log \frac{1}{k(n)}  \sum_{j=1}^{k(n)}  \exp \big (  \scal{\lambda}{S_n(j)}\big )\Big )=\Lambda(\lambda)=\|\lambda\|^2/2 
$$
for a dense and countable subset of points $\lambda\in B$, hence for all $\lambda\in B$ by convexity of the functions $\Lambda^\omega_n$. This is enough to get the result.

\subsection{Proof of Theorem~\ref{brownfonc}}\label{pfbrown2} We let the reader adapt the lines of the proof of Theorem~\ref{corASLDP} to the present situation. The only change is that here for each $n\ge 1$ one must consider the i.i.d sequence of Brownian motions obtained by juxtaposition of the $k(n)$ sequences of $n$ Brownian motions $\big ((W_{k(n),j,i})_{t\in [0,1]}\big )_{1\le i\le n}$, $1\le j\le k(n)$, where $W_{k(n),j,i}(t)=(nk(n))^{1/2} \big (W(\frac{j-1}{k(n)}+\frac{i-1+ t}{nk(n)})- W(\frac{j-1}{k(n)}+\frac{i-1}{nk(n)})\big )$, so that $W_{k(n),j}/n^{1/2}=S_n(j)/n$ with $S_n(j)=\sum_{i=1}^n W_{k(n),j,i}$.

\end{document}